\documentclass[a4paper,12pt]{article}
\usepackage[utf8]{inputenc}
\usepackage[english]{babel}  
\usepackage[T1]{fontenc}      
\usepackage{geometry}         
\usepackage{amsfonts}
\usepackage{amsmath}
\usepackage{amssymb}
\usepackage[pdftex]{graphicx}
\usepackage{graphics}
\usepackage{float}
\usepackage{color}
\usepackage{tikz}
\usepackage{amsthm}
\usepackage{dsfont}
\usepackage{fourier}

\def\R{\mathop{\mathbb{R}}\nolimits}

\def\ksi {\mathop{\xi}\nolimits}

\def\* {\mathop{\otimes}\nolimits}
\def\+{\mathop{\oplus}\nolimits}

\def \Z{\mathop{\mathbb{Z}}\nolimits}

\def\lt{\left}
\def\rt{\right}
\def\[{[\![}
\def\]{]\!]}

\DeclareMathOperator{\Supp}{Supp}

\newtheorem{theorem}{Theorem}[section]
\newtheorem{corollary}[theorem]{Corollary}
\newtheorem{lemma}[theorem]{Lemma}
\newtheorem{conjecture}[theorem]{Conjecture}
\newtheorem{proposition}[theorem]{Proposition}
\newtheorem*{proposition*}{Proposition}
\newtheorem{definition}{Definition}[section]
\newtheorem*{theorem*}{Theorem}
\newtheorem*{lemma*}{Lemma}
\newtheorem*{corollary*}{Corollary}

\theoremstyle{remark}
\newtheorem{remark}{Remark}

\title{From the Newton equation to the wave equation~: the case of shock waves}
\author{Xavier Blanc\footnote{Univ. Paris Diderot, Sorbonne Paris Cit\'e, Laboratoire Jacques-Louis Lions, UMR 7598, UPMC, CNRS, F-75205 Paris, France }~ \& Marc Josien\footnote{CERMICS, Ecole Nationale des Ponts et Chauss\'ees, 6-8 avenue Blaise Pascal, Cit\'e Descartes, 77455 Marne-la-Vall\'ee}~~\footnote{INRIA, Paris-Rocquencourt}.}
\begin{document}

\maketitle

\begin{abstract}
    We study the macroscopic limit of a chain of atoms governed by the Newton equation. It is known from the work of Blanc, Le Bris, Lions, that this limit is the solution of a nonlinear wave equation, as long as this solution remains smooth. We show, numerically and mathematically that, if the distances between particles remain bounded, it is not the case any more when there are shocks -at least for a convex nearest-neighbour interaction potential with convex derivative.
\end{abstract}

\section{Introduction}
  \paragraph{Motivation} We investigate here the macroscopic limit of the time-dependent Newton equation ruling the evolution of a set of particles at the microscopic scale. 
  We perform our study in a simplified context: the particles form a one-dimension chain and we suppose that the interactions between the particles are nearest-neighbour interactions.
  It has been proven in \cite{BLL} that, when the potential is convex, this system tends to a wave equation, provided that the solution of this wave equation is regular. However, non-linear wave equations are known to develop shocks in finite time. 
  Our aim is to examine how this phenomenon impacts the convergence of Newton equations to wave equation.\\
  Consider $2N$ particles, indexed by $j \in \[-N,N-1\]$ and with positions $X_j$ which interact through the Newton equation, for $j \in \[-N+1,N-2\]$:
  \begin{equation}\label{Newton}
    \frac{d^2}{dt}X_j(t) = W'\lt(X_{j+1}(t)-X_j(t)\rt) - W'\lt(X_j(t)-X_{j-1}(t)\rt), 
  \end{equation}
  where $W$ is the interaction potential.
  Throughout the article, we assume that $W$ is even.
  The initial and boundary conditions are:
  \begin{align}\label{IDD}
    &X_{j+1}-X_j(0)=\phi_0^x\lt(\frac{j}{N}\rt) \text{ and } \frac{d}{dt}X_j(0)=\phi^\tau_0\lt(\frac{j}{N}\rt), \\
    &X_{-N}(t)=N\phi_l \text{ and } X_{N}(t)=N\phi_r. \label{BCD}
  \end{align}
  We introduce the following rescaling:
  \begin{align*}
    t=N\tau, &&
    j=Nx.
  \end{align*}
  The time $t$ is the microscopic time while $\tau$ is the macroscopic time. Then, the semi-discrete equation \eqref{Newton} is consistent with the wave equation:
  \begin{equation}\label{WE}
    \partial_\tau^2 \phi(\tau,x)=\partial_x \lt[W'(\partial_x \phi(\tau,x))\rt],
  \end{equation}
  with initial and boundary conditions:
  \begin{align}\label{IDC}
    &\partial_x \phi(\tau=0,x)=\phi^x_0(x) \text{ and } \partial_\tau \phi(\tau=0,x)=\phi^\tau_0(x),\\
    &\phi(\tau,-1)=\phi_l \text{ and } \phi(\tau,1)=\phi_r. \label{BCC}
  \end{align}
  
  \begin{remark}
    It is worth pointing out that the natural variables in the hyperbolic system~\eqref{WE} are $\partial_\tau \phi$ and $\partial_x \phi$, in the sense that:
    \begin{equation*}
      \partial_\tau 
      \lt(
	\begin{array}{c}
	  \partial_\tau \phi\\
	  \partial_x \phi
	\end{array}
      \rt)
      =
      \partial_x
      \lt(
	\begin{array}{c}
	  W'\lt(\partial_x \phi\rt)\\
	  \partial_\tau \phi
	\end{array}
      \rt),
    \end{equation*}
    which is a $p$-system (see \cite{Serre}, p 127-131). We therefore introduce their discrete analogues:
    \begin{align*}
      &U_j=X_{j+1}-X_j && \text{ and } &&V_j=\frac{d X_j}{dt}.
    \end{align*}
  \end{remark}

    \begin{remark}[About inversion]\label{RqInversion}
      One could a priori think that \eqref{Newton} may lead to some inversions of atom positions, especially when shocks occur (see \cite{brenier_rearrange}). Put differently, one could have $X_{j+1}(t)<X_j(t)$ for certain $t$ and $j$, even if $X_j(t=0)$ was increasing. This would question the physical relevance of \eqref{Newton}, for the $j$-th particle is supposed to interact with its nearest neighbours (which are the $j-1$-th and the $j+1$-th particles if and only if $X_j$ is monotone). However, numerical simulations show that, for many interesting initial conditions (including many of those that lead to shocks), such inversions never occur. We therefore assume throughout the article that condition $X_j(t)<X_{j+1}(t)$ holds for all $t, j$.
    \end{remark}

  In the regular case and if $W$ is convex, it has been proven in \cite{BLL} that \eqref{Newton} converges to \eqref{WE} in the following sense:
  
  \begin{theorem}\label{ThBLL}
    Assume that $W \in \mathcal{C}^4(\R)$, and that $W'' \geq \alpha >0$. Suppose $\phi_l=-1$ and $\phi_r=1$. Assume that $\phi \in \mathcal{C}_\tau\lt(\mathcal{C}^4_x\rt)$ is a solution to \eqref{WE} for the initial and boundary conditions \eqref{IDC} and \eqref{BCC}. Let $X_j(t)$ be the unique solution to \eqref{Newton} for the initial and boundary conditions \eqref{IDD} and \eqref{BCD}. Then we have the following convergences:
    \begin{align}\label{BLL1}
      &\forall \tau \in [0,T[, & \sup_{-N \leq i \leq N-1} 
      \lt| \frac{1}{N} X_j(N\tau) - \phi\lt(\tau, \frac{i}{N}\rt) \rt| \underset{N \rightarrow \infty}{\rightarrow} 0, \\
      \label{BLL2}
      & \forall \tau \in [0,T[, &  \sup_{-N\leq i \leq N-1} 
      \lt| \frac{dX_j}{dt}(N\tau) - \partial_\tau \phi \lt(\tau, \frac{i}{N}\rt) \rt| 
      \underset{N \rightarrow \infty}{\rightarrow} 0.
    \end{align}
  \end{theorem}
  
  \begin{remark}
    Theorem \ref{ThBLL} is not stated in \cite{BLL} for the initial condition \eqref{IDD}, but with:
    \begin{align*}
      X_j(t=0)= N\int_{-1}^{j/N} \phi_0^x(x) dx && \text{and} && \frac{d X_j}{dt}(t=0)=\phi_0^\tau\lt(\frac{j}{N}\rt).
    \end{align*}
    As easily seen, its proof however also applies to the initial condition \eqref{IDD}.
  \end{remark}

  When $W$ is convex but not quadratic, even if $\phi_0^x$ and $\phi_0^\tau$ are smooth, shocks generally occur in finite time for solutions of \eqref{WE}. By shock, we mean that the solution $\phi$ of \eqref{WE} becomes irregular (see \cite{Serre} for examples). An interesting question is what happens after such shocks for the discrete system \eqref{Newton}, and in particular if there is still a link between \eqref{Newton} and~\eqref{WE}. To answer this question, we will consider Riemann-like initial conditions, as is customary in the study of hyperbolic systems.\\
  Let us underline that \eqref{Newton}, which can be seen as a semi-discrete numerical scheme, is taken for granted, as it comes from a physical model.
  Some authors take the opposite way, and modify given schemes (adding viscosity for example) in order to go from the discrete system to the continuous one (see \cite{Mielke}), or to help the numerical computation of hyperbolic systems (\cite{tadmornumerical}).\\
  Let us also mention that their exists a quite detailed study on discrete systems ruled by \eqref{Newton} in the particular case where:
  \begin{equation}\label{TodaPot}
    W(u)=\exp(-u).
  \end{equation}
  In that case, called the Toda lattice (see  \cite{McLaughlin}, \cite{Flaschka}, \cite{Toda}, \cite{deiftvernakides}), the discrete Hamiltonian system is completely integrable: this allows for a detailed description of the solutions. It is well-known that \eqref{WE} does not describe well the limiting system and that the solutions are dispersive waves. This is linked with Lax pairs, and helps to make the connection with the Korteweg-de Vries equation (see \cite{LaxLevermore1}). 
  We will not investigate in this article this particular case, which is, in our understanding, closely linked with the special structure induced by the potential \eqref{TodaPot}. We shall however demonstrate that the solutions associated with more general potentials globally  display the same features as the dispersive waves of the Toda lattice (see Section \ref{SecNLin}).

  \paragraph{Numerics} In order to have a better understanding of \eqref{Newton}, we perform some numerical experiments. To do so, we use a Verlet scheme (see \cite{lebris}, p 111) on the variables $U_j$ and $Z_j:=\frac{d U_j}{dt}$.
  More explicitly, we simulate:
  \begin{equation}\label{SchemaVerlet}
    \lt\{
      \begin{array}{l}
	      {U}_j^{n+1/2}={U}_j^n+\frac{\delta t}{2} {Z}_j^n,\\[6pt]
	      {Z}_{j}^{n+1/2}={Z}_j^n+\frac{\delta t}{2} \lt(W'\lt({U}_{j+1}^{n+1/2}\rt)-2W'\lt({U}_{j}^{n+1/2}\rt)+W'\lt({U}_{j-1}^{n+1/2}\rt) \rt),\\[6pt]
	      {U}_j^{n+1}={U}_j^n+\delta t{Z}^{n+1/2}_j,\\[6pt]
	      {Z}_j^{n+1}={Z}_j^n+\delta t\lt(W'\lt({U}_{j+1}^{n+1/2}\rt)-2W'\lt({U}_{j}^{n+1/2}\rt)+W'\lt({U}_{j-1}^{n+1/2}\rt) \rt),
      \end{array}
    \rt.
  \end{equation}
  where $X^n_j$ is an approximation for $X_j(n\delta t)$.
  We take an initial condition corresponding to a Riemann problem or a smooth initial condition that develops shocks in finite time (for the sake of simplicity, we only use Riemann problems for illustrations in this article). 
  The crucial feature of \eqref{SchemaVerlet} is that it preserves the Hamiltonian properties of \eqref{Newton} (for \eqref{SchemaVerlet} is symplectic).
  The error we make on $U_j$ in $L^2$ norm is of order $O(NT\delta t)$ (see \cite{Hairer} p13), where $T$ is the final macroscopic time of simulation, which allows to simulate \eqref{Newton} for a reasonably large number $2N$ of particles ($N \simeq 10^4$), and thus to have a fair experimental knowledge of the system \eqref{Newton}.
  
  \paragraph{Outline of the article}

    In Section \ref{SecNota}, we introduce the notations and collect some classical facts about \eqref{Newton} and \eqref{WE}. In particular, we focus on the initial and boundary conditions, that are supposed to mimic the Riemann problem. We also focus on the natural energy of these systems.\\
    In Section \ref{SecRes}, we state and next illustrate our main results. We focus first on the simple quadratic potential $W(u)=u^2/2$ and claim that the convergence of \eqref{Newton} to \eqref{WE} is true for a large class of initial conditions. This is proved in Section \ref{SecLin}. 
    Then we examine the case where both $W$ and $W'$ are strongly convex. We show that, if the distances between neighbouring particles remain bounded and if the energy of the continuous system \eqref{WE} is not preserved, solutions of \eqref{Newton} do not converge to solutions of \eqref{WE}. It is based on the fact that the system \eqref{Newton} displays the property of light cone: the perturbations propagate with a finite speed at macroscopic level. This is proved in Section \ref{SecNLin}.
    We state next a conjecture about a uniform bound on the distances between particles of the system \eqref{Newton}, that we justify with numerics and that we question through a study of the linear case. This conjecture is motivated by the fact that the assumption of boundedness of the distances between particle is a major assumption in every result of Section \ref{SecNLin}. We discuss it in Section \ref{SecConjec}.
    Finally, we state that discrete shock waves do not exist, either when $W(u)=\frac{u^2}{2}$ or when $W'$ and $W''$ are strictly convex. It is proved in Section \ref{SecSoliton}.

  \section{Preliminaries}\label{SecNota}
 
     \subsection{General notations}
      Let $q \in [1, \infty]$. For $Y_j$, with $j \in [\![-N,N-1]\!]$, we denote by:
      \begin{equation*}
	\lt\| \lt(Y_j\rt) \rt\|_{l_j^q}=
	\lt\{
	  \begin{array}{l l}
	    \lt(\sum_{i=-N}^{N-1} Y_j^q\rt)^{1/q}  & \text{if }q<\infty,\\
	    \max_{j \in [\![-N,N-1]\!]} \lt|Y_j\rt|& \text{if } q=\infty.
	  \end{array}
	\rt.
      \end{equation*}
      We denote by $\mathcal{C}_{p}$ the set of piecewise continuous functions on $[-1,1]$, and $\mathcal{C}^1_{p}$ the set of piecewise continuous functions on $[-1,1]$ that have piecewise continuous derivatives. We use the subscript ${per}$ for functional spaces to indicate that we intersect these spaces with the space of $2$-periodic functions. We use the subscripts $x$, $\tau$, $t$ for functional spaces to indicate that these spaces have their variables $x$ in $[-1,1]$, $\tau$  in $[0,T]$,  respectively $t$ in $[0,NT]$. For example:
      \begin{align*}
        H_\tau^1\lt( \mathcal{C}_{x}\rt):=H^1\lt([0,T],\mathcal{C}([-1,1])\rt).\\
      \end{align*}
    \subsection{Initial data and boundary conditions}
      In the present article, we mainly use Dirichlet boundary conditions. They have the advantage of being consistent with Riemann problems. In Section \ref{SecConjec}, we will also use periodic boundary conditions for technical reasons;  more specifically, when the potential $W$ is quadratic, it allows for an explicit resolution of \eqref{Newton}.\\
      We say that \eqref{Newton} (respectively \eqref{WE}) is set with Dirichlet boundary conditions if \eqref{IDD} and \eqref{BCD} (respectively \eqref{IDC} and \eqref{BCC}) are satisfied, with $\phi^x_0, \phi^\tau_0 \in \mathcal{C}_p$ and  $\phi_l, \phi_r  \in \R$ being compatible in the following sense:
      \begin{align}
	&\int_{-1}^1 \phi^x_0(x)dx=\phi_r-\phi_l,
	&\phi^\tau_0(-1)= 0, && \phi^\tau_0(1)=0. \label{DirSet1}
      \end{align}
      We say that the system \eqref{Newton} is set with periodic boundary conditions if \eqref{Newton} is satisfied for all $j$ with the convention that $X_{N+j}=X_{-N+j}$. The associated initial conditions are~\eqref{BCD} with $\phi_0^\tau, \phi^x_0 \in \mathcal{C}_{p}$ such that the compatibility condition:
        \begin{equation*}
          \int_{-1}^1 \phi^x_1(x) dt = 0
        \end{equation*}
      is satisfied.

    \subsection{Hypotheses on $W$}
    
      We suppose that $W$ is $\mathcal{C}^2$ and strongly convex:
      \begin{equation}\label{Convex}
        W''\geq \alpha >0.
      \end{equation}
      Indeed, this assumption implies that \eqref{WE} is a strictly hyperbolic system (if not, the theory for \eqref{WE} is far more complex).
      A very particular case is when $W$ is quadratic:
      \begin{equation}
	\label{PotQuad}
	W(u)=\frac{u^2}{2}.
      \end{equation}
      When we consider a non-quadratic potential, we also assume that $W$ is $\mathcal{C}^3$ and that $W'$ is strictly convex:
      \begin{align}
	\label{SConvex}
	&W''' >0.
      \end{align}
      Let us emphasize that \eqref{WE} is genuinely non-linear when $W'''>0$ (see \cite{Serre} p~113 and p~127). We speak about the \textit{linear case} (respectively the \textit{nonlinear case}) when \eqref{PotQuad} is satisfied (respectively when \eqref{Convex} and \eqref{SConvex} are satisfied). The terminology may seem ambiguous, but it is justified by \eqref{WE}, which involves $W'$ and not $W$.\\
      The convexity \eqref{Convex}, and \textit{a fortiori} \eqref{SConvex}, is obviously a strong and non-physical simplification, as a physical potential should be even (and non-constant even potentials with other minima than $0$ cannot satisfy \eqref{Convex} on $\R$). For example, our results do not directly cover this ``quadratic'' potential:
      \begin{equation}\label{PotDur}
        W(u)=\lt(|u| -1 \rt)^2.
      \end{equation}
      Our numerical experiments suggest that for given initial conditions, the distances between particles is bounded from below and from above (see Remark \ref{RqInversion} and Section~\ref{SecConjec}). Hence one can require \eqref{Convex} or \eqref{SConvex} to be true only on the corresponding intervals. For example, if we know a priori that the order of the particles is preserved, one can apply our results with the potential \eqref{PotDur}.

    \subsection{The discrete system}
      \paragraph{Notations}
      For the discrete system \eqref{Newton}, we denote:
      \begin{align*}
	V_j(t)=\frac{dX_j}{dt}(t), &&
	U_j(t)=X_{j+1}(t)-X_j(t), &&
	Z_j(t)=\frac{dU_j}{dt}(t).
      \end{align*}
      \begin{remark}[Dependence on $N$]
	$X_j$ and the other discrete quantities implicitly depend on $N$. When necessary, we write $X_j^N$, $U^N_j$, \textit{et cetera}.
      \end{remark}
      The correspondence between the discrete system and the continuous system in encoded in the following notations:
      \begin{align*}
	&k^N(x):=\lfloor Nx \rfloor,\\
	&\theta^N(x):= Nx-k^N(x),\\
	&\phi^N(\tau,x):=\frac{1-\theta^N(x)}{N} X_{k^N(x)}(t)+\frac{\theta^N(x)}{N} X_{k^N(x)+1}(t),\\
	& \zeta^N(\tau,x):=\partial_\tau \phi^N(\tau,x)=\lt( \lt(1-\theta^N(x) \rt)\frac{d}{dt} X_{k^N(x) }+\theta^N(x)\frac{d}{dt}X_{k^N(x) +1} \rt)(t),\\
	&\ksi^N(\tau,x):=\frac{d}{dt} X_{k^N(x)}(t).
      \end{align*}
      Remark that $\partial_\tau \phi^N$ is the linear interpolation of $V_j$, and is therefore not equal to $\ksi^N(\tau,x)$, which corresponds to $V_j(t)$.
      In any case, we extend the functions $\phi^N, \ksi^N$ by continuity with constant branches on $]-\infty,-1] \cup [1,+\infty[$. For example, we have $\phi(x)=\phi_l$ if $x<-1$.
      
      \paragraph{Properties of the discrete system}
      
      The discrete system \eqref{Newton} is an Hamiltonian system, with the energy:
      \begin{equation}\label{ED}
	\mathcal{E}_D(t):=\frac{1}{2N} \sum_{j=-N}^{N-1} V_j^2(t) + \frac{1}{N} \sum_{j=-N}^{N-1} W(U_j(t)).
      \end{equation}
      The energy \eqref{ED} is the total mechanical energy of the system. The kinetic energy is the first term and the potential energy is the second term.
      Either in Dirichlet or in periodic setting, an elementary calculation shows that the discrete energy $\mathcal{E}_D$ is preserved:
      \begin{equation}\label{EnergDis}
        \frac{d}{dt}\mathcal{E}_D(t)=0.
      \end{equation} 
      As a consequence, the energy \eqref{ED} being convex, a direct application of the Cauchy-Lipschitz theorem implies that~\eqref{Newton} has a unique solution for every time $t \in [0,+\infty[$, provided $W$ satisfies \eqref{Convex}. For later purpose, we define the notion of discrete compatibility.
 
      \begin{definition}[$D$-compatibility]
        We say that $T>0$ is $D$-compatible with $\phi_0^x$ and $\phi_0^\tau$ if there exist $\delta>0$ and $C>0$ such that:
        \begin{align*}
	  &\lt|\partial_x \phi^N(\tau,x)-\phi_0^x(-1) \rt| \leq CN^{-1} &&  \forall (\tau,x) \in [0,T]\times [-1,-1+\delta],\\
	  & \lt|\partial_x \phi^N(\tau,x)-\phi_0^x(1) \rt| \leq CN^{-1} && \forall (\tau,x) \in [0,T] \times [1-\delta,1],
	\end{align*}	      
	and $\partial_\tau \phi^N$ converges uniformly to $0$ on  $[0,T] \times \lt\{ [-1,-1+\delta] \cup [1-\delta,1] \rt\}$, as $N$ goes to infinity.
      \end{definition}
      $D$-compatibility means that the solution of \eqref{Newton} is almost not perturbed near the boundary $x=-1$ and $x=1$, until time $T$.

    \subsection{The continuous system}
    
      Let $T>0$. Following \cite{Serre}, p 28, we say that $\phi \in W^{1,\infty}_{\tau,x}$ is a weak solution of \eqref{WE} with initial conditions \eqref{IDC} if, for all $g \in \mathcal{C}^{\infty}_c\lt( ]-\infty,T[ \times ]-1,1[\rt)$:
      \begin{align}
	\label{weakf}
	&\int_0^T \int_{-1}^1 \lt\{\partial_x g W'\lt(\partial_x \phi\rt)- \partial_\tau g \partial_\tau \phi\rt\}(\tau,x) dx d\tau = \int_{-1}^1  g(0,x)\phi^\tau_0(x) dx,\\
	&\int_0^T \int_{-1}^1 \lt\{\partial_x g\partial_\tau \phi -\partial_\tau g \partial_x \phi \rt\}(\tau,x) dx d\tau = \int_{-1}^1 g(0,x) \phi_0^x(x) dx.
	\label{weakf2}
      \end{align}
      We say that a weak solution $\phi$  of \eqref{WE} is an entropy solution if it also satisfies in the weak sense (see \cite{Serre}, p 82):
      \begin{equation}\label{EntropySol}
	\frac{d}{d\tau} \mathcal{E}_C(\tau) \leq 0,
      \end{equation}
      where $\mathcal{E}_C$ is the continuous energy associated with $\phi$:
      \begin{equation}
        \label{EnergCont}
        \mathcal{E}_C(\tau) =  \int_{-1}^1 \lt\{\frac{1}{2}\lt(\partial_\tau \phi\rt)^2 + W\lt(\partial_x \phi\rt) \rt\}(\tau,x) dx.
      \end{equation} 
      We are interested in weak entropy solutions $\phi$ of \eqref{WE} satisfying
      \begin{equation}
	\label{LoseEntropy}
	\mathcal{E}_C(T) < \mathcal{E}_C(0).
      \end{equation}
      Shocks satisfy \eqref{LoseEntropy}. We recall now the definition of the Riemann problems:
      \begin{definition}[Riemann problem]
	Let $u_l, u_r, v_l, v_r \in \R$, and:
	\begin{align}\label{Rpb}
	  \phi_0^x(x) := 
	  \lt\{
	    \begin{array}{l l}
	      u_l & \text{if } x<0,\\
	      u_r & \text{if } x \geq 0,
	    \end{array}
	  \rt. &&
	  \phi^\tau_0(x) := 
	  \lt\{
	    \begin{array}{l l}
	      v_l & \text{if } x<0,\\
	      v_r & \text{if } x \geq 0.
	    \end{array}
	  \rt.
	\end{align}
	Solving the Riemann problem associated with $(u_l, u_r, v_l, v_r)$ consists in finding  $\phi$ an entropy solution of \eqref{WE} on $[0,T] \times\R$ with initial conditions \eqref{IDC}.
      \end{definition}
      This is the classical Riemann problem. However, it is possible to use weaker assumptions on $\phi^x_0$ and $\phi^\tau_0$, that simulate what we call a \emph{boundary} Riemann problem. This second definition is more flexible and allows to work with a very large class of initial conditions (for example, smooth initial data that develop discontinuities in finite times, in system \eqref{WE}). Namely:
      \begin{definition} [Boundary Riemann problem] Let $u_l, u_r \in \R^2$, and:
	\begin{align}
	  \label{Riemannbord}
	  \phi_0^x(x) := 
	  \lt\{
	    \begin{array}{l l}
	      u_l & \text{if } x <-1/2,\\
	      u_r & \text{if } x >1/2,
	    \end{array}
	  \rt. &&
	  \phi^\tau_0(x) := 
	  \lt\{
	    \begin{array}{l l}
	      0 & \text{if } x <-1/2,\\
	      0 & \text{if } x >1/2,
	    \end{array}
	  \rt.
	\end{align}
	without further requirement on $\phi^x_0$ and $\phi^\tau_0$ between $-1/2$ and $1/2$. Solving this boundary Riemann problem consists in finding $\phi(\tau,x)$, entropy solution of \eqref{WE} with initial and boundary conditions \eqref{IDC} and \eqref{BCC}, and $X_j(t)$, solution of \eqref{Newton} with initial and  boundary conditions  \eqref{IDD} and \eqref{BCD}, with $\phi_l$ and $\phi_r$ being constant so that \eqref{DirSet1} is satisfied.
      \end{definition}
      We impose $\phi^\tau_0$ to vanish near the boundary in the boundary Riemann problem \eqref{Riemannbord} so that $\phi_l$ and $\phi_r$ are constant; this is useful to avoid some technicalities about boundary conditions.\\
      The solutions of the Riemann problem \eqref{Rpb} are combinations of rarefaction waves and shock waves. One does not change the solution of \eqref{WE} (for $T$ sufficiently small) if one restricts $\phi$ to $x \in [-1,1]$ and solves \eqref{WE} with Dirichlet boundary conditions \eqref{BCC}.\\
      For example (see \cite{Serre}, p~127-131), if $u_r>u_l$ and if the following Rankine-Hugoniot condition is satisfied:
      \begin{align}\label{RH}
	v_r-v_l = \sigma \lt(u_r-u_l\rt) && \text{and} && W'(u_r)-W'(u_l)=\sigma(v_r-v_l), 
      \end{align}
      then the entropy solution of the Riemann problem reads as:
      \begin{align*}
	&\partial_x\phi(\tau,x)=\lt\{ 
	  \begin{array}{l l} 
	    u_l & \text{ if } x<-\sigma \tau \\
	    u_r & \text{ if }  -\sigma \tau \leq x
	  \end{array} 
	  \right., &&
	  &\partial_\tau \phi(0,x)=\lt\{
	  \begin{array}{l l} 
	    v_l & \text{ if } x<-\sigma \tau \\
	    v_r & \text{ if } -\sigma \tau\leq x
	  \end{array},
	\rt.
      \end{align*}
      and satisfies \eqref{LoseEntropy}.
      We are interested in boundary Riemann problems. As a consequence, we focus on weak solutions $\phi \in W^{1,\infty}_{\tau,x}$ of \eqref{WE} in the Dirichlet setting that can be continued by a constant on the right and on the left:
      
      \begin{definition}[$C$-compatibility]
        Let $\phi^x_0$ and $\phi^\tau_0 \in \mathcal{C}_p$, $\phi_l$ and $\phi_r \in \R$ satisfying \eqref{DirSet1}. Assume that $\phi$ is an entropy solution of \eqref{WE} on $[0,T]\times[-1,1]$ with initial and boundary conditions \eqref{IDC} and \eqref{BCC}. We say that $T$ is $C$-compatible with $\phi^x_0$ and $\phi^\tau_0$ if there exists $\delta>0$ such that:
        \begin{equation}\label{R-Compatible}
	  \phi([0,T]\times[-1,-1+\delta])=\lt\{\phi_l\rt\} \text{ and } \phi([0,T]\times[1-\delta,1])=\lt\{\phi_r\rt\}.
	\end{equation}        
      \end{definition}
      If $T$ is $D$-compatible and $C$-compatible with $\phi^x_0$ and $\phi^\tau_0$, we say that it is $DC$-compatible. Basically, $DC$-compatibility provides a strong control on the solutions of \eqref{Newton} and \eqref{WE} near the boundary $x=-1$ and $x=1$, until time $T$.
      
      For the linear system, we have the following theorem of existence and uniqueness (see Theorem 3 p 384 and Theorem 4 p 385 of \cite{Evans}):
      \begin{theorem}[Existence and uniqueness in the linear case]\label{ExistUneLin}
	Let $\phi_0^x, \phi^\tau_x \in L^2_x$.
	Suppose that $a, b \in \mathcal{C}^1_x$, and $a\geq a_0>0$.
	Then, there exists one and only one solution $\phi \in H^1_{\tau,x}$ to:
	\begin{equation*}
	  \partial_\tau^2 \phi(\tau,x) = \partial_x \lt(a(x) \partial_x \phi(\tau,x) + b(x) \rt),
	\end{equation*}
	with initial and boundary conditions \eqref{IDC} and \eqref{BCC}. In addition, the energy $\tilde{\mathcal{E}}_C$ of the system is preserved:
	\begin{equation*}
	  \tilde{\mathcal{E}}_C= \int_{-1}^1 \lt\{\frac{1}{2}\lt(\partial_\tau \phi(\tau,x)\rt)^2 + \frac{a(x)}{2} \lt(\partial_x \phi(\tau,x)\rt)^2+b(x) \partial_x \phi(\tau,x) \rt\}dx.
	\end{equation*}
      \end{theorem}
      It is clear that this energy extends the above definition \eqref{EnergCont}.
    
    \subsection{Discrete shock waves}
    
      \begin{definition}\label{DiscShock}
	We say that $X_j(t)$, $j \in \Z, t\in\R_+$, is a discrete shock wave of \eqref{Newton} associated with $(u_l,u_r) \in \R^2$, $u_l \neq u_r$, if $X_j(t)$ satisfies \eqref{Newton} and if there exist $\phi \in \mathcal{C}^2(\R)$ and $c \in \R$ such that:
	\begin{align*}
	  &X_j(t)=\phi(j-ct) && \forall t \in \R_+, \forall j \in \Z,\\
	  &\lim_{x \rightarrow - \infty} \phi'(x)=u_l,\\
	  & \lim_{x \rightarrow + \infty} \phi'(x)=u_r.
	\end{align*}
      \end{definition}
      The definition implies that:
      \begin{align}\label{Soliton}
	c^2 \phi''(x)=W'(\phi(x+1)-\phi(x))-W'(\phi(x)-\phi(x-1)).
      \end{align}

  \section{Results}\label{SecRes}
    We state here our main results and illustrate them with some numerical results.
    
    \subsection{The linear case}
      When $W(u)=u^2/2$, one observes that $\phi^N$ converges in $H^1_{\tau,x}$ to $\phi$. One can even see that for regular initial conditions, this convergence seems to hold in every $W^{1,p}_{\tau,x}$. This is illustrated by Figure \ref{figfig1}:
      \begin{figure}[H]
	\caption{Comparison between $\partial_x \phi^N$ (black curve) and $\partial_x \phi$ (red curve) for Riemann shock initial conditions. $W(u)=\frac{u^2}{2}$, $N=10000$, $\tau=0.3$. }\label{figfig1}
	\begin{center}
	  \includegraphics[scale=0.7]{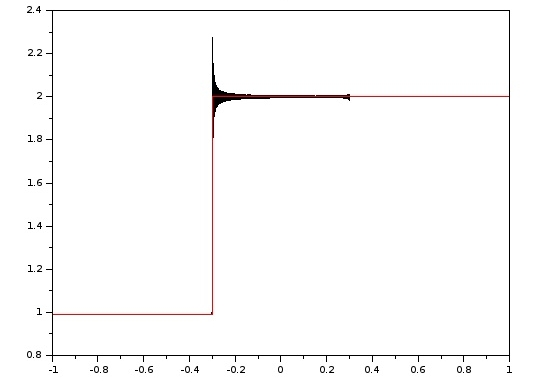}
        \end{center}
      \end{figure}
      We prove this convergence in a generalized framework, where the quadratic potential $W$ depends not only on $u$ but also on $x$:
      
      \begin{theorem}\label{ThmConvLin2}
	Let $T>0$, $\phi_l, \phi_r \in \R$, $\phi_0^x$ and $\phi^\tau_0 \in \mathcal{C}_p$ satisfy \eqref{DirSet1}. Assume that:
	\begin{equation*}
	  W(x,u)=\frac{1}{2}A(x)u^2+B(x)u.
	\end{equation*}
	with $A, B \in \mathcal{C}^1_x$ and $A$ satisfying:
	\begin{equation}\label{Ellipticitee}
	  A\geq \alpha>0,
	\end{equation}
	for $\alpha$ a given positive constant.
	Consider the solution $X^N_j(t)$ to:
	\begin{equation}
	  \label{Newton2}
	  \frac{d^2}{dt^2} X^N_j(t) = \partial_u W\lt(\frac{j}{N},U^N_j(t)\rt)-\partial_u W\lt(\frac{j-1}{N},U^N_{j-1}(t)\rt)
	\end{equation} 
	for all $j \in \[-N+1,N-2\]$ for the initial and boundary conditions \eqref{IDD} and \eqref{BCD}. Then the associated $\phi^N$ converges:
	\begin{equation*}
	  \phi^N \underset{N \rightarrow+\infty}{\rightarrow} \phi \text{ in } H^1_{\tau,x},
	\end{equation*} 
	where $\phi$ is the unique solution of:
	\begin{equation}\label{WE2}
	  \partial_\tau^2\phi(\tau,x)=\partial_x \lt( \partial_u W(x,\partial_x \phi(\tau,x)\rt),
	\end{equation}
	for the initial and boundary conditions \eqref{IDC} and \eqref{BCC}.
      \end{theorem}
      
      It has a direct corollary:
      \begin{corollary}\label{ThmConvLin}
        Let $T>0$, $\phi_l, \phi_r \in \R$, $\phi_0^x$ and $\phi^\tau_0 \in \mathcal{C}_p$ satisfying \eqref{DirSet1}. Assume that $W$ satisfies \eqref{PotQuad}. Let $\phi$ be the solution of \eqref{WE} for the initial and boundary conditions \eqref{IDC} and~\eqref{BCC}, and $X^N_j$ be the solution of \eqref{Newton} for the initial and boundary conditions \eqref{IDD} and~\eqref{BCD}.
        We have the following convergence:
        \begin{equation*}
          \phi^N \underset{N \rightarrow+\infty}{\rightarrow} \phi \text{ in } H^1_{\tau,x}.
        \end{equation*}
      \end{corollary} 
      
      \begin{remark}[Less restrictive assumptions]
        For both Theorem \ref{ThmConvLin2} and Corollary \ref{ThmConvLin}, it is sufficient to assume that $\phi_0^x$ and $\phi_0^\tau \in L^2_x$ as long as $X_j^N$ satisfies the initial condition:
        \begin{align*}
	  &X_{j+1}(0)-X_j(0)=\partial_x \phi^N\lt(\tau=0,\frac{j}{N}\rt) \text{ and } \frac{d}{dt}X_j(0)=\partial_\tau \phi^N\lt(\tau=0,\frac{j}{N}\rt),
	\end{align*}
        such that:
        \begin{align*}
          &\partial_x \phi^N(\tau=0,.) \rightarrow \phi^x_0 \text{ in } L^2_x, &&   
          \ksi^N(\tau=0,.) \rightarrow \phi^\tau_0 \text{ in } L^2_x.
        \end{align*}
      \end{remark}

    \subsection{The non-linear case}

      If the potential $W$ is convex but not quadratic, when there is a shock, we observe on numerical simulations that $\partial_x \phi^N$ does not converge strongly to the associated $\partial_x \phi$. It does not even converge weakly. Actually $\partial_x \phi^N$ oscillates with a high frequency and an amplitude that does not decrease when $N$ grows. We believe that this situation is generic for basically any potential such that $W'$ is not affine on the zone where $U_j$ evolves. We can to prove this non-convergence, under the extra-hypothesis that $W'$ is strictly convex, and under the assumption that the distance between particles $U_j^N$ is bounded uniformly in $N$, $j$ and $t \in [0,NT]$ (the latter assumption is discussed in Section \ref{SecConjec}).\\
      We illustrate this non-convergence with Figure \ref{figfig3} comparing $\partial_x \phi^N$ and $\partial_x \phi$. Remarkably enough, even if there are large oscillations, let us remark that distances between particles remain bounded. We check numerically that $\partial_x \phi^N$ coincides with $\partial_x \phi$ outside a region of space away from the shock that grows linearly in macroscopic time; this property is known for Toda lattice \cite{Flaschka}. It is also known that the Korteweg-de-Vries equation \cite{LaxLevermore1} has a similar behaviour.
      \begin{figure}[H]
	\caption{Comparison between $\partial_x \phi^N$ (black curve) and  $\partial_x \phi$ (red curve) for Riemann shock initial conditions, and a non-linear potential. $N=5000$, $\tau=0.075$, $W(u)=\frac{u^6}{6}$.}\label{figfig3}
        \begin{center}
          \includegraphics[scale=1]{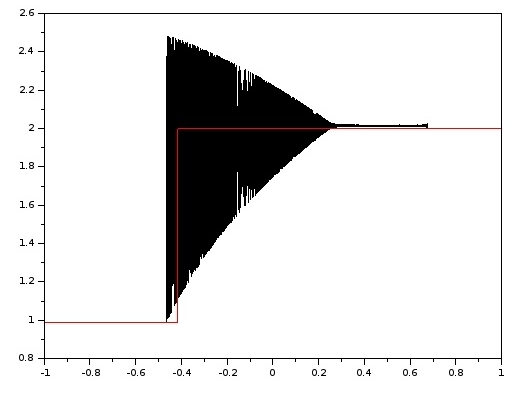}
        \end{center}
      \end{figure}

      \begin{theorem}\label{ThmcvNLin}
	Let $W \in \mathcal{C}^3([a,b])$ satisfy \eqref{Convex} and \eqref{SConvex}. Assume that $\phi^x_0, \phi^\tau_0 \in \mathcal{C}^1_{p,x}$ and $\phi_l, \phi_r \in \R$ satisfy \eqref{DirSet1} and \eqref{Riemannbord}, for $u_l, u_r \in \R$. Let $T_0>0$.\\
	Let $X_j^N(t)$ be the solution of \eqref{Newton} for the initial and boundary conditions \eqref{IDD} and \eqref{BCD}. Suppose that:
	\begin{align}\label{BorneLinfty}
	  U_j^N(t) \in [a,b] && \forall j \in \[-N,N-1\], \forall N>0,  \forall t \in [0,NT_0].
	\end{align}
	Then there exists a $D$-compatible $T\leq T_0$.\\      
	Assume that $\phi \in W^{1,\infty}_{\tau,x}$, is an entropy solution of \eqref{WE} for the initial and boundary conditions \eqref{IDC} and \eqref{BCC}, that $T$ is $C$-compatible and that there exists $T_1<T$ satisfying:
	\begin{equation}\label{PertEc}
	  \mathcal{E}_C(T_1)<\mathcal{E}_C(0).
	\end{equation}
	Then $\phi^N$ does \textbf{not} converge to $\phi$ in the sense of distribution in space and time $D'_{\tau,x}$.
      \end{theorem}

      \begin{remark}[Entropy solution]
	In Theorem \ref{ThmcvNLin}, we only compare $\phi^N$ to the \textit{entropy}  solution $\phi$ of \eqref{WE}. 
	We think that $\phi^N$ cannot converge to \textit{any} weak solution of \eqref{WE}.
	Indeed, if $\phi^N$ converges to $\phi$, which is a solution of \eqref{WE}, Lemma \ref{LemRenforce} below implies that $\partial_x \phi^N$ converges strongly to $\partial_x \phi$.
	But numerical experiments show that $\partial_x \phi^N$ oscillates too much so that it cannot converge strongly to anything: this justifies our conclusion.
      \end{remark}
      
      \begin{remark}[Reversibility]
	\eqref{Newton}~is a reversible system, but \eqref{WE} is not when shocks occur, whereas both systems are reversible as long as the solution $\phi$ of \eqref{WE} remains smooth enough. In the first case, the discrete system does not converge to the continuous one (Theorem \ref{ThmcvNLin}), but it does in the second case (Theorem \ref{ThBLL}).
      \end{remark}

      \begin{remark}[Convergence breakdown]
	Suppose that $W$ satisfies \eqref{Convex} and \eqref{SConvex}.
	Let $\phi^x_0$ and $\phi^\tau_0$ be smooth functions. Define $X_j^N$ and $\phi$ as in Theorem \ref{ThmcvNLin}.
	Suppose that there exists a $DC$-compatible $T>0$ such that $\mathcal{E}_C(T)<\mathcal{E}_C(0)$ -in other words, a shock occurs.
	If Conjecture~\ref{ConjecDur} above holds for initial data $\phi^x_0$ and $\phi^\tau_0$, it leads to the following paradoxical situation:
	\begin{enumerate}
	  \item{until a certain time $T_0<T$, $\phi(\tau,.)$ is sufficiently smooth, so that Theorem \ref{ThBLL} applies. Thus:
	  \begin{equation*}
	    \phi^N \rightarrow \phi \text{ in } \mathcal{C}\lt([0,T_0]\times[-1,1] \rt),
	  \end{equation*}}
	  \item{applying Theorem \ref{ThmcvNLin}, we get that $\phi^N$ does not converge to $\phi$ in $D'([0,T_1] \times[-1,1])$ as soon as $\mathcal{E}_C(T_1)< \mathcal{E}_C(0)$.}
	\end{enumerate}
	Therefore, shocks break the discrete-to-continuum convergence of \eqref{Newton} to \eqref{WE}.
      \end{remark}
      
      It is immediate to see that \eqref{Newton} instantly propagates perturbations in the discrete system. It can be proven by linearizing \eqref{Newton} and assuming a small perturbation $\epsilon$ on a fixed $j_0$-th particle:
      \begin{align*}
        \tilde{X}_{j}(0) = X_{j}(0)+\epsilon \delta_{j_0}^j, && \frac{d X_j}{dt}(0) = \frac{d \tilde{X}_j}{dt}(0).
      \end{align*}
      We assume that both $X_j$ and $\tilde{X}_j$ satisfy \eqref{Newton}. Integrating iteratively \eqref{Newton} for small time $\Delta t$, we get (for $j>0$) at leading order in $\epsilon$  (the proof of this formal expansion is in the spirit of the proof of Proposition \ref{proplightcone} below):
      \begin{align*}
        \tilde{X}_{j_0+j}(\Delta t)-X_{j_0+j}(\Delta t )\simeq \epsilon \frac{(\Delta t)^{2j}}{(2j)!} \prod_{k=0}^{j-1} W''\lt(U_{j_0+k}(0)\rt).
      \end{align*}
      However, on the macroscopic level, this propagation has a finite speed. This paradox is due to the fact that the influence of perturbation on $x_0$ at $t_0$ decays exponentially outside a cone $|x-x_0|\leq c \lt|t-t_0\rt|$. It is noticeable that this light cone property is an important feature of hyperbolic systems. It is however a key ingredient to prove that~\eqref{Newton} does \textit{not} converge to \eqref{WE}.
      
      We formalize the fact that perturbations of the discrete system propagate with a finite speed on the macroscopic level by the following theorem:
	
      \begin{theorem}\label{ThLum}
        Let $W \in \mathcal{C}^2(\R)$ satisfy \eqref{Convex}. Let $T>0$. Assume that $X^N_j(t)$ and $\tilde{X}^N_j(t)$ satisfy \eqref{Newton}, for $j \in \[0,N-1\]$, $t \in [0,NT]$, with right boundary condition $\tilde{X}^N_N=X^N_N=N\phi_r$. We denote by:
        \begin{align*}
          \tilde{V}^N_j(t)=\frac{d}{dt}\tilde{X}^N_j(t), && \tilde{U}^N_j(t):= \tilde{X}^N_{j+1}(t)-\tilde{X}^N_j(t).
        \end{align*}
        Suppose that 
        \begin{align*}
          X^N_j(t=0)=\tilde{X}^N_j(t=0),&  \forall j \in \[1,N-1\],\\
          V^N_j(t=0)=\tilde{V^N_j}(t=0),& \forall j \in \[1,N-1\].
        \end{align*}
        Assume that there exists $C \in \R_+$ such that, $\forall N>0$:
        \begin{equation}\label{DEC}
          \sup_{t \in \R_+} \lt( \sum_{j=0}^{N-1} W\lt(U^N_j(t)\rt) + \frac{1}{2} \sum_{j=0}^{N-1} \lt(V^N_j(t)\rt)^2 \rt) \leq C N,
        \end{equation}
	and:
        \begin{align}
          &K=\sup_{u \in ]u_1,u_2[} \lt| W''(u) \rt| < +\infty \label{HypCone1},          
        \end{align}
        where:
        \begin{align*}
          &u_1:=\inf_{ \scriptsize{\begin{array}{l l } N>0,\\ j \in \[-N,N-1\],\\ t \in [0,NT] \end{array}}}
          \lt\{\min \lt(U^N_j(t),\tilde{U}^N_j(t) \rt) \rt\},\\
          &u_2:=\sup_{ \scriptsize{\begin{array}{l l } N>0,\\ j \in \[-N,N-1\],\\ t \in [0,NT] \end{array}}}
          \lt\{\max \lt(U^N_j(t),\tilde{U}^N_j(t) \rt) \rt\}.
        \end{align*}
	Let $c=\exp(2) \sqrt{K}$. Let $x \in ]0,1[$. Then, for all $\tau < \frac{x}{c}$, we have:
        \begin{align}
	  \label{ThLUmU}
          &\lim_{N\rightarrow + \infty} \lt\{N \sup_{\scriptsize{\begin{array}{l} 0\leq t < N\tau,\\ j > Nx \end{array}}} 
          \lt| U_j^N(t) - \tilde{U}_j^N(t) \rt| \rt\} =0,\\
          \label{ThLumV}
          &\lim_{N\rightarrow + \infty} \sup_{\scriptsize{\begin{array}{l} 0\leq t < N\tau,\\ j > Nx \end{array}}}
          \lt|V_j^N(t)-\tilde{V}^N_j(t) \rt|=0.
        \end{align}
      \end{theorem}

      \begin{figure}[H]
	\caption{Light cone on the surface $\lt(t,U_j^N(t)\rt)$, for Riemann initial conditions corresponding to a shock wave, for a non-linear potential $W(u)=\frac{u^6}{6}$}
	\begin{minipage}{0.25\linewidth}
	  \begin{center}
	    \begin{tikzpicture}[scale=2]
	      \coordinate (O) at (0,0);
	      \coordinate (A) at (0,1);
	      \coordinate (B) at (1,0);
	      \coordinate (C) at ({sqrt(2)/2},{sqrt(2)/2});
	      \draw [->] (O)--(A);
	      \draw [->] (O)--(B);
	      \draw [->] (O)--(C);
	      \draw (B) node [below] {$j$};
	      \draw (A) node [left] {$U_j$};
	      \draw (C) node [right] {$t$};
	    \end{tikzpicture}
	  \end{center}
	\end{minipage}
	\begin{minipage}{0.7\linewidth}
	  \begin{center}
	    \includegraphics[scale=0.3]{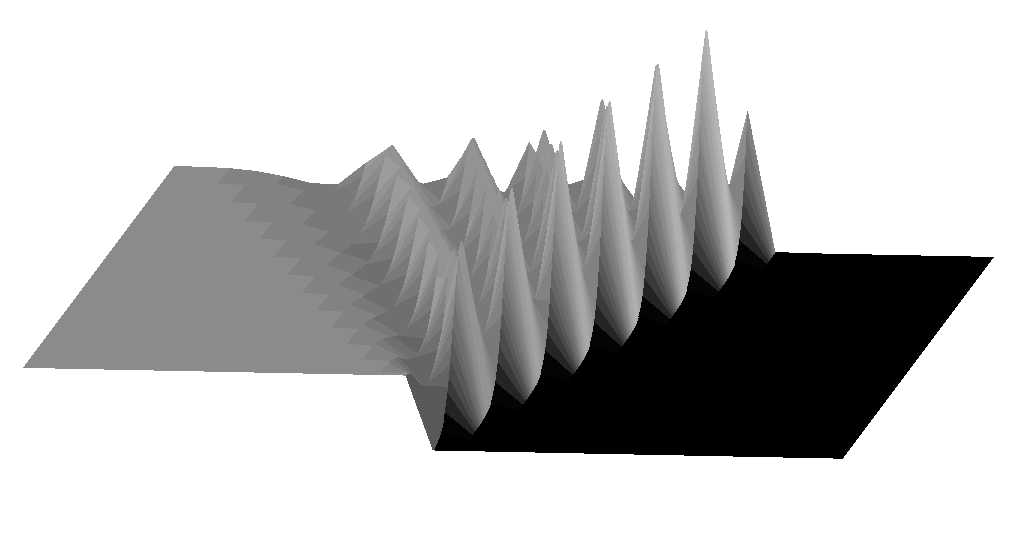}
	  \end{center}
	\end{minipage}
      \end{figure}
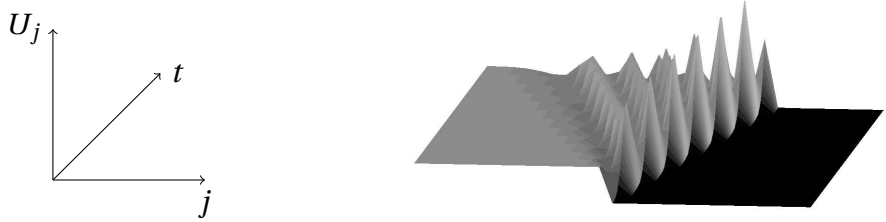
      
      A few remarks are in order:
      \begin{remark}
        The speed $c$ in Theorem \ref{ThLum} is not optimal -we see it from numerical experiments- but it has the same order as the natural speed of \eqref{WE}, given by \eqref{RH}. Formally:
        \begin{equation*}
          \exp(2)\sqrt{K} \propto \sup_{u \in [u_l,u_r]} \sqrt{|W''(u)|} \propto \sqrt{\lt|\frac{W'(u_r)-W'(u_l)}{u_r-u_l} \rt|}.
        \end{equation*}
        The above formal calculation is exact when $W$ is quadratic.
      \end{remark}

      \begin{remark}
        Assumption \eqref{HypCone1} of Theorem \ref{ThLum} is automatically fulfilled if there exists $\alpha, \beta \in \R$ such that $\alpha\leq W''(u) \leq \beta, \forall u \in \R$. This is the case when the potential $W$ is quadratic. However, such a bound cannot hold if $W'$ is strongly convex: one needs to know a priori that the distances $U_j^N$ between particles is bounded uniformly in $N$.
      \end{remark}
      
  \subsection{Uniform $L^\infty_t\lt(l_j^\infty\rt)$ bound}
  
      Most of the results we are able to prove in the non-linear case require the assumption that, for given initial data, the distance between particles remains bounded uniformly in $N$ (that is, \eqref{BorneLinfty} is satisfied). We have not been able to prove that this assumption is fulfilled. We formulate the following conjecture:
      \begin{conjecture}\label{ConjecDur}
	Suppose that $W\in \mathcal{C}^2(\R)$ satisfies \eqref{Convex}. Assume that $\phi_0^x$ and $\phi^\tau_0 \in \mathcal{C}^1_{p,x}$. Then, there exist $T>0$ and $a<b \in \R$ depending only on $W$, $\phi_0^x$ and $\phi_0^\tau$ such that, for $X_j(t)$ satisfying \eqref{Newton} for the initial and boundary conditions \eqref{IDD} and \eqref{BCD}:
	\begin{align}
	  \label{ConjecDur1}
	  a \leq U_j^N(N\tau) \leq b, &&\forall j \in \[-N,N-1\], \forall N \in \mathbb{N}, \forall \tau \in [0,T].
	\end{align} 
      \end{conjecture}
      
      Note that in the case of Riemann problem \eqref{Rpb}, the initial conditions satisfy the hypotheses of Conjecture \ref{ConjecDur}.
      We checked Conjecture \ref{ConjecDur} numerically for a large set of piecewise smooth initial data, with potentials of the form $W(u)=Au^\gamma+Bu^2$, $\gamma>2$, $A, B \in \R_+$. 
      When $\phi$ is sufficiently smooth, Conjecture \ref{ConjecDur} can be proven (by Theorem~\ref{ThBLL}).

      Let us point out the fact that it seems necessary to require some smoothness on the initial conditions in Conjecture \ref{ConjecDur}. In other words, one cannot hope that, for $X^N_j$ solution of \eqref{Newton}, for given $T>0$, $\lt\|U^N_j\rt\|_{l^\infty_{j,t}}$ is controlled by $\lt\|U^N_j(t=0)\rt\|_{l^\infty_{j}}$ and $\lt\|V^N_j(t=0)\rt\|_{l^\infty_{j}}$ uniformly in $N$.       
      Indeed, we can prove the following proposition:
      \begin{proposition}\label{ConjSubtl}
        Let $W(u)=\frac{u^2}{2}$, and $\tau_0>0$. There exists a sequence of initial conditions $U_j^N(t=0)$, $V^N_j(t=0)$ such that, for $X^N_j(t)$ the corresponding solutions of \eqref{Newton} with periodic boundary conditions, we have:
        \begin{align*}
          \lt\|U_j^N(t=0)\rt\|_{l^\infty_j} \leq 1,
          &&\lt\|V_j^N(t=0)\rt\|_{l^\infty_j}\leq 1,
          &&\lt\|U_j^N(N\tau_0)\rt\|_{l^\infty_j}\underset{N\rightarrow\infty}{\rightarrow} +\infty.
        \end{align*}
      \end{proposition}
      The following plot shows the explosion:
      \begin{figure}[H]\caption{Successive pictures of $U_j(N\tau)$. $N=1000$, $\tau=0,0.08,...,0.4$}
	    \begin{center}
	     \includegraphics[scale=0.2]{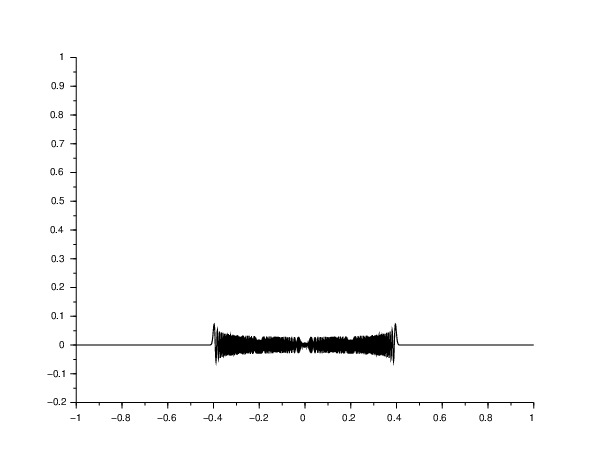}
	     \includegraphics[scale=0.2]{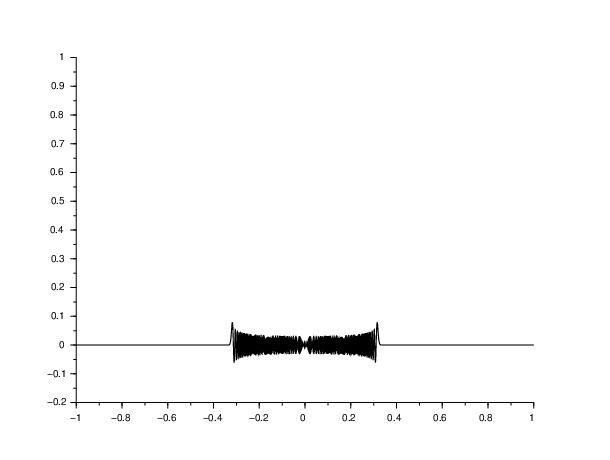}
	     \includegraphics[scale=0.2]{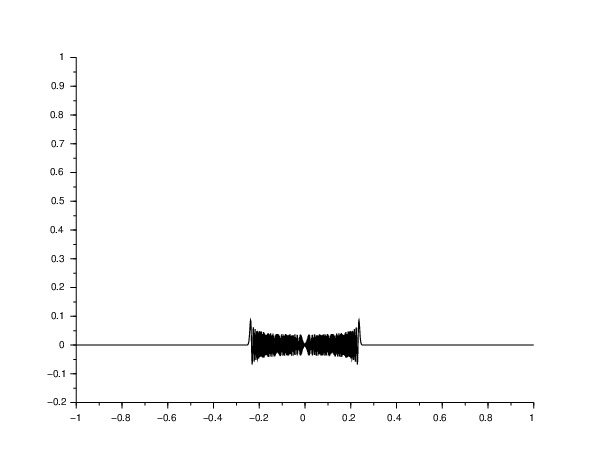}
	     \includegraphics[scale=0.2]{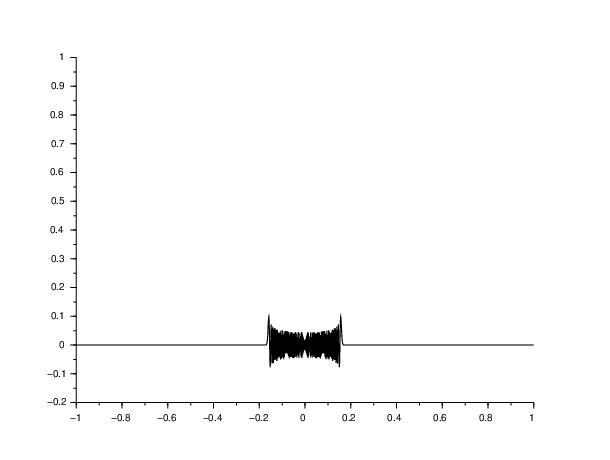}
	     \includegraphics[scale=0.2]{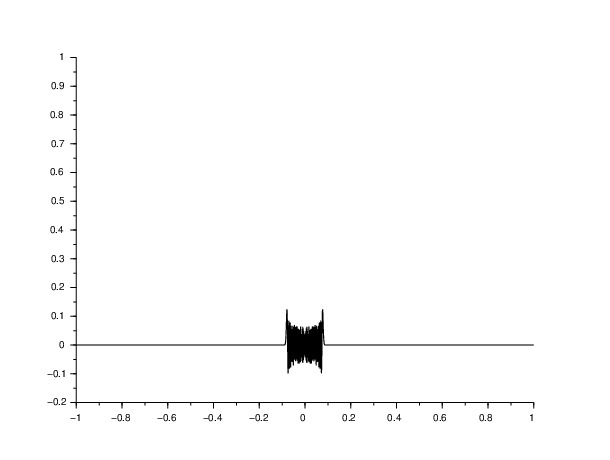}
	     \includegraphics[scale=0.2]{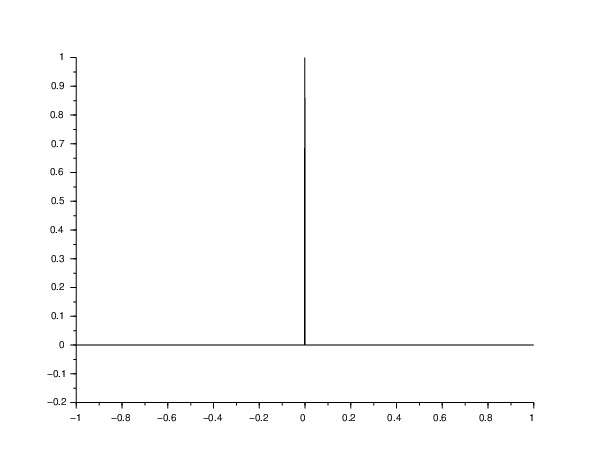}
	    \end{center}
      \end{figure}

    \subsection{Non-existence of discrete shock waves}
    
      A natural question is whether or not there exist non-trivial discrete shock waves for the Newton equation \eqref{Newton}. Should such discrete progressive waves exist, one could expect that they would describe an important feature of the limit of \eqref{Newton} system when $N \rightarrow +\infty$. Unfortunately, we prove that discrete shock waves do not exist, even when the potential is quadratic. More specifically, we prove the following propositions:
      
      \begin{proposition}\label{DiscreteSolitonsLin}
        Suppose $W$ satisfies \eqref{PotQuad}. Then there exist no discrete shock wave for~\eqref{Newton}, in the sense of Definition \ref{DiscShock}.
      \end{proposition}

      \begin{proposition}\label{DiscreteSolitonsNLin}
        Suppose $W$ satisfies \eqref{Convex} and \eqref{SConvex}. Then there exist no discrete shock wave for~\eqref{Newton}, in the sense of Definition \ref{DiscShock}.
      \end{proposition}
      
      \begin{remark}
        It is straightforward from the proof that there does not exist any other discrete wave than the constant ones in the linear case. In the non-linear case, we do not know if there exists solitons, that is $X_j(t)$ satisfying Definition \ref{DiscShock}, with the slight modification that $u_l=u_r$.
      \end{remark}

  \section{The linear case}\label{SecLin}
    
    When the potential is quadratic,
    the corresponding wave equation \eqref{WE} is linear. Its characteristic lines do not cross, therefore, when the initial conditions are regular, shocks never occur. Furthermore, the energy is preserved: the continuous system \eqref{WE} is thus conservative, as the discrete one \eqref{Newton}. This is the reason why the discrete system naturally tends to the continuous one, and we show it with simple arguments, essentially using weak compactness of $H^1_{\tau,x}$. This extends the result of \cite{BLL}.\\

    Let us prove Theorem \ref{ThmConvLin2}. We first prove that the discrete energy is preserved:
      
      \begin{lemma}\label{ConservEnergy}
        Under the hypotheses of Theorem \ref{ThmConvLin2}, the following generalized discrete energy is preserved: 
        \begin{equation}
          \tilde{\mathcal{E}}_D^N(t):=\sum_{k=-N}^{N-1} W\lt(\frac{k}{N},U_k(t) \rt) + \frac{1}{2} \sum_{k=-N}^{N-1} \lt(V_k(t)\rt)^2.
        \end{equation} 
      \end{lemma}
      
      \begin{proof}
        Using \eqref{Newton}, we get that:
        \begin{align*}
          \frac{d}{dt} \tilde{\mathcal{E}}_D^N(t)
          =&\sum_{k=-N}^{N-1}\lt\{ \partial_u W\lt(\frac{k}{N},U_k \rt) \lt(V_{k+1} - V_k \rt)\rt\}(t) + \sum_{k=-N+1}^{N-1} \lt\{V_k \frac{d^2}{dt^2}X_k\rt\}(t),\\
          =&\sum_{k=-N}^{N-1}\lt\{ \partial_u W\lt(\frac{k}{N},U_k \rt) \lt(V_{k+1} - V_k \rt)\rt\}(t) \\
          &+ \sum_{k=-N+1}^{N-1} \lt\{V_k  \lt( \partial_u W\lt(\frac{k}{N},U_k \rt)-\partial_u W\lt(\frac{k-1}{N},U_{k-1} \rt) \rt)\rt\}(t).
        \end{align*}
        Reorganizing the sum, we obtain:
        \begin{align*}
          \frac{d}{dt} \tilde{\mathcal{E}}_D^N(t)=\lt\{\partial_u W \lt(\frac{N-1}{N},U_{N-1} \rt)V_N -\partial_u W \lt(-1,U_{-N} \rt)V_{-N}\rt\}(t).
        \end{align*}
        Yet, as \eqref{BCD} is satisfied, we have $V_N=V_{-N}=0$. Therefore:
        \begin{equation*}
          \frac{d}{dt} \tilde{\mathcal{E}}^N_D(t)=0,
        \end{equation*}
        which implies the desired result.
      \end{proof}
      
      Next we prove that \eqref{Newton2} is consistent with \eqref{WE2}:
      
      \begin{lemma}\label{LemConsistLin}
        Under the hypotheses of Theorem \ref{ThmConvLin2}, we have the following convergences
        for all $g \in \mathcal{C}^{\infty}_c\lt( ]-\infty,T[ \times ]-1,1[\rt)$:
	\begin{align}
	  \nonumber
	  &\int_0^T \int_{-1}^1 \lt\{\partial_x g  \partial_uW\lt(x,\partial_x \phi^N \rt)- \partial_\tau g \partial_\tau \phi^N\rt\}(\tau,x) dx d\tau\\
	  & - \int_{-1}^1  g(0,x)\phi^\tau_0(x) dx \underset{N\rightarrow +\infty}{\rightarrow} 0,\label{Conv11}\\
	  &\int_0^T \int_{-1}^1 \lt\{\partial_x g\partial_\tau \phi^N -\partial_\tau g \partial_x \phi^N \rt\}(\tau,x) dx d\tau - \int_{-1}^1 g(0,x) \phi_0^x(x) dx \underset{N\rightarrow +\infty}{\rightarrow} 0.\label{Conv12}
	\end{align}
      \end{lemma}
      
      \begin{proof}
        It is easy to prove \eqref{Conv12} by an integration by parts:
        \begin{align*}
          &\int_0^T \int_{-1}^1 \lt\{\partial_x g\partial_\tau \phi^N -\partial_\tau g \partial_x \phi^N \rt\}(\tau,x) dx d\tau\\
          &=-\int_0^T \int_{-1}^1 \lt\{g \partial_{x} \partial_\tau \phi^N - g \partial_x \partial_\tau \phi^N \rt\} (\tau,x) dx d\tau + \int_{-1}^1 g(0,x) \partial_x\phi^N(0,x) dx\\
          & \underset{N\rightarrow +\infty}{\rightarrow} \int_{-1}^1 g(0,x) \phi^x_0(x) dx.
        \end{align*}
        Before proving \eqref{Conv11}, let us introduce the operators:
        \begin{align*}
          D_N^- : f(x) \mapsto N\lt( f(x) - f\lt(x-\frac{1}{N}\rt) \rt) && D_N^+ : f(x) \mapsto N\lt( f(x) - f\lt(x+\frac{1}{N}\rt) \rt),
        \end{align*}
        which are adjoint of each other.\\
        We integrate by parts:
        \begin{align}
          \nonumber
          &\int_0^T \int_{-1}^1 \lt\{\partial_x g \partial_u W \lt(x,\partial_x \phi^N\rt)- \partial_\tau g \partial_\tau \phi^N\rt\}(\tau,x) dx d\tau\\
          \nonumber
          &=\int_0^T \int_{-1}^1 \lt\{ \partial_x g  \partial_u W \lt( x, \partial_x \phi^N\rt)+  g \partial^2_\tau \phi^N\rt\}(\tau,x) dx d\tau \\
          &~~~+ \int_{-1}^1 g(0,x) \partial_\tau \phi^N(0,x) dx.
	  \label{ConvLin11}
        \end{align}
        It is clear that:
        \begin{equation}\label{ConvLin12}
          \int_{-1}^1 g(0,x) \partial_\tau \phi^N(0,x) dx \underset{N \rightarrow + \infty}{\rightarrow} \int_{-1}^1 g(0,x) \phi^\tau_0(x) dx.
        \end{equation}
        We focus on the other integrals. By definition, if $x \in [-1+1/N,1-1/N]$:
        \begin{align*}
           \partial_\tau^2 \phi^N(\tau,x)
	    &=N \frac{d^2}{dt^2} \lt( \lt(1-\theta^N(x)\rt) X_{k^N(x)} + \theta^N(x) X_{k^N(x)+1}\rt)(t) \\
	    &=\lt(1-\theta^N(x)\rt)D_N^-\lt\{ A\lt(\frac{k^N(x)}{N} \rt) U_{k^N(x)}(t)
	    +B\lt( \frac{k^N(x)}{N} \rt)\rt\}\\
	    &+ \theta^N(x) D_N^-\lt\{A \lt( \frac{k^N(x) +1}{N} \rt)  U_{k^N(x)+1}(t) + B \lt(\frac{k^N(x)+1 }{N}\rt) \rt\}\\
	    &=\lt(1-\theta^N(x)\rt)D_N^-\lt\{ A\lt(\frac{k^N(x)}{N} \rt) \partial_x \phi^N(\tau,x)
	    +B\lt( \frac{k^N(x)}{N} \rt)\rt\}\\
	    &+ \theta^N(x) D_N^-\lt\{A \lt( \frac{k^N(x) +1}{N} \rt)  \partial_x \phi^N(\tau,x+1/N) + B \lt(\frac{k^N(x)+1}{N}\rt) \rt\}.
        \end{align*}
        Remark that $D_N^{\pm}$ and $\theta^N$ commute, in the sense that: 
        \begin{equation*}
          D_N^{\pm}\lt\{\theta^N(x)f(x)\rt\}=\theta^N(x)D_N^{\pm}\lt\{f(x)\rt\}.
        \end{equation*}
        Hence:
        \begin{align*}
           \partial_\tau^2 \phi^N(\tau,x)
	    &=D_N^-\Bigg\{
	    \lt(1-\theta^N(x)\rt)\lt( A\lt(\frac{k^N(x)}{N} \rt) \partial_x \phi^N(\tau,x)
	    +B\lt( \frac{k^N(x)}{N} \rt)\rt)\\
	    &~~~+ \theta^N(x) \lt(A \lt( \frac{k^N(x) +1}{N} \rt)  \partial_x \phi^N(\tau,x+1/N) + B \lt(\frac{k^N(x)+1}{N}\rt) \rt) 
	    \Bigg\},
	\end{align*}
	if $|x|<1-1/N$. We assume that $N$ is sufficiently large, so that $\Supp(g) \subset [-1+2/N,1-2/N]$. Then:
        \begin{align*}
	    &\int_0^T \int_{-1}^1 \partial_\tau^2 \phi^N(\tau,x) g(\tau,x) dx d\tau \\
	    &=\int_0^T \int_{-1}^1 g(\tau,x) 
	    D_N^-\Bigg\{
	    \lt(1-\theta^N(x)\rt)\lt( A\lt(\frac{k^N(x)}{N} \rt) \partial_x \phi^N(\tau,x)
	    +B\lt( \frac{k^N(x)}{N} \rt)\rt)\\
	    &~~~+ \theta^N(x) \lt(A \lt( \frac{k^N(x) +1}{N} \rt)  \partial_x \phi^N(\tau,x+1/N) + B \lt(\frac{k^N(x)+1}{N}\rt) \rt) 
	    \Bigg\}
	    dx d\tau.
	\end{align*}
	As $D_N^-$ and $D_N^+$ are adjoint of each other:
	\begin{align*}
	   &\int_0^T \int_{-1}^1 \partial_\tau^2 \phi^N(\tau,x) g(\tau,x) dx d\tau\\
	   &=\int_0^T \int_{-1}^1 D_N^+\lt\{g(\tau,x)\rt\} 
	   \Bigg\{
	   \lt(1-\theta^N(x)\rt)\lt( A\lt(\frac{k^N(x)}{N} \rt) \partial_x \phi^N(\tau,x)
	   +B\lt( \frac{k^N(x)}{N} \rt)\rt)\\
	   &~~~+ \theta^N(x) \lt(A \lt( \frac{k^N(x+1/N) }{N} \rt)  \partial_x \phi^N(\tau,x+1/N) + B \lt(\frac{k^N(x+1/N)}{N}\rt) \rt) 
	   \Bigg\}
	   dx d\tau\\
	   &=\int_0^T \int_{-1}^1 \lt[ \lt(1-\theta^N(x) \rt) D_N^+\lt\{g(\tau,x)\rt\}+\theta^N(x) D_N^+ \lt\{g(\tau,x-1/N)\rt\} \rt]\\
	   &~~~\lt( A\lt( \frac{k^N(x)}{N} \rt)\partial_x \phi^N(\tau,x) + B\lt(\frac{k^N(x)}{N}\rt) \rt) dx d\tau.
	\end{align*}
	Now, since $A$, $B$ and $g$ are regular:
	\begin{align*}
	  &D_N^+\lt\{g(\tau,x)\rt\}=-\partial_x g(\tau,x)+ \frac{h_g^N(\tau,x)}{N},\\
	  &A\lt( \frac{k^N(x)}{N} \rt)=A(x)+\frac{h_A^N(x)}{N},\\
	  &B\lt(\frac{k^N(x)}{N}\rt)=B(x)+\frac{h_B^N(x)}{N},
	\end{align*}
	where:
	\begin{equation*}
	  \sup_{N} \sup_{\tau,x}  \lt\{ \lt| h_g^N(\tau,x) \rt|+\lt| h_A^N(\tau,x) \rt|+\lt| h_B^N(\tau,x) \rt| \rt\} < +\infty.
	\end{equation*}
	Therefore:
	\begin{align}
	  \nonumber
	  &\int_0^T \int_{-1}^1 \partial_\tau^2 \phi^N(\tau,x) g(\tau,x) dx d\tau\\
	  &=-\int_0^T \int_{-1}^1 \partial_x g(\tau,x) \lt(A(x) \partial_x \phi^N(\tau,x)+B(x) \rt) dx d\tau + C^N,
	  \label{LinDemo1}
	\end{align}
	where, by the Cauchy-Schwarz inequality:
	\begin{equation}\label{LinDemo2}
	  C^N \leq \frac{C}{\sqrt{N}} \lt(1 + \sqrt{\int_{0}^T \int_{-1}^1 \lt(\partial_x \phi^N(\tau,x)\rt)^2 dx d\tau } \rt).
	\end{equation}
	Using Lemma \ref{ConservEnergy} and \eqref{Ellipticitee}, we get that:
	\begin{equation}
	  \int_{-1}^1 \lt(\partial_x \phi^N(\tau,x)\rt)^2 dx \leq C. \label{LinDemo3}
	\end{equation}
	Whence, from \eqref{LinDemo1}, \eqref{LinDemo2} and \eqref{LinDemo3} we deduce:
	\begin{equation}\label{LinDemo4}
	  \lt|\int_0^T \int_{-1}^1 \lt\{\partial_\tau^2 \phi^N g + \partial_x g \partial_u \lt(W(x,\partial_x \phi^N \rt)  \rt\}(\tau,x) dx d\tau \rt| \underset{N \rightarrow +\infty}{\rightarrow} 0.
	\end{equation} 
	From  \eqref{ConvLin11}, \eqref{ConvLin12} and \eqref{LinDemo4}, we obtain \eqref{Conv11}.
      \end{proof}
      
      We are now able to prove Theorem \ref{ThmConvLin2}.
      
      \begin{proof}
	Using Lemma \ref{ConservEnergy} and \eqref{Ellipticitee}, we estimate:
	\begin{align*}
	  \sum_{k=-N}^{N-1} \lt\{ A\lt(\frac{k}{N} \rt)U_k^2 + V_k^2 \rt\}(t)
	  \overset{\eqref{Ellipticitee}}{\leq}& C + C\sum_{k=-N}^{N-1} \lt\{ A\lt(\frac{k}{N} \rt)\lt(U_k\rt)^2+B\lt(\frac{k}{N}\rt) U_k +\lt(V_k\rt)^2 \rt\}(t)   \\
	  \leq &C\lt(1+ \tilde{\mathcal{E}}_D^N(t)\rt)\\
	  \leq&C \lt(1+ \tilde{\mathcal{E}}^N_D(0) \rt).
	\end{align*}
	Since $\phi_0^x$ and $\phi_0^\tau$ are sufficiently regular, we have:
	\begin{equation*}
	  \mathcal{E}_D(0) \underset{N\rightarrow\infty}{\rightarrow} \int_{-1}^1 W\lt(x,\phi^x_0(x)\rt) dx + \frac{1}{2} \int_{-1}^1 \lt(\phi^\tau_0(x)\rt)^2dx.
	\end{equation*}
	We therefore obtain that, for all $\tau>0$:
	\begin{equation*}
	  \int_{-1}^1 \lt\{A\lt(\frac{k^N(x)}{N} \rt) \lt(\partial_x \phi^N(\tau,x)\rt)^2+ \lt(\partial_\tau \phi^N(\tau,x) \rt)^2  \rt\}dx \leq C.
	\end{equation*}
	And by smoothness of $A$ and thanks to the fact that $A\geq \alpha>0$, we obtain:
	\begin{equation}\label{Produitscal}
	  \int_{-1}^1 \lt\{A(x) \lt(\partial_x \phi^N(\tau,x)\rt)^2+ \lt(\partial_\tau \phi^N(\tau,x) \rt)^2  \rt\}dx \leq C.
	\end{equation}
	We take the following scalar product on $\dot{H}^1_{\tau,x}$:
	\begin{equation*}
	  \lt<g,h\rt>_{\tilde{\dot{H}}^1_{\tau,x}}:= \int_0^T \int_{-1}^1 \lt\{A(x) \partial_x g\partial_x h+\partial_\tau g \partial_\tau h \rt\}(\tau,x) dx d\tau.
	\end{equation*}
	This scalar product induces a norm which is equivalent to the classical one on $\dot{H}^1_{\tau,x}$, as $A\geq \alpha>0$ and $A$ is bounded. We denote $\tilde{\dot{H}}^1_{\tau,x}$ for $\dot{H}^1_{\tau,x}$ endowed with this scalar product, respectively $\tilde{H}^1_{\tau,x}$ for $H^1_{\tau,x}$ endowed with the scalar product:
	\begin{equation*}
	  \lt<g,h\rt>_{\tilde{H}^1_{\tau,x}}=\lt<g,h\rt>_{\tilde{\dot{H}}^1_{\tau,x}} + \int_0^T \int_{-1}^1 g(\tau,x)h(\tau,x) dx d\tau.
	\end{equation*}
	From \eqref{Produitscal} and the fact that $\phi^N(-1)=\phi_l$, we get by the Poincaré inequality that $\phi^N$ is bounded in $\tilde{H}^1_{\tau,x}$. By weak compactness of this space, we extract:
	\begin{equation}
	  \phi^N \rightharpoonup \phi_\infty \text{ in } \tilde{H}^1_{\tau,x}.
	\end{equation} 
	Lemma \ref{LemConsistLin} implies that, for all $g \in \mathcal{C}^\infty_c([0,T[,]-1,1[)$:
	\begin{align*}
	  &\int_0^T \int_{-1}^1 \lt\{\partial_x g  \partial_uW\lt(x,\partial_x \phi^\infty \rt)- \partial_\tau g \partial_\tau \phi_\infty\rt\}(\tau,x) dx d\tau = \int_{-1}^1  g(0,x)\phi^\tau_0(x) dx,\\
	  &\int_0^T \int_{-1}^1 \lt\{\partial_x g\partial_\tau \phi_\infty -\partial_\tau g \partial_x \phi_\infty \rt\}(\tau,x) dx d\tau =\int_{-1}^1 g(0,x) \phi_0^x(x) dx.
	\end{align*}
	Therefore, $\phi_\infty$ is $\phi$, the unique solution of \eqref{WE} for the initial and boundary conditions \eqref{IDC} and \eqref{BCC} given by Theorem \ref{ExistUneLin}.\\
	We now prove that this convergence is strong. As $\tilde{\mathcal{E}}^N_D$ is preserved, we have:
	\begin{align*}
	  \lt\|\phi^N\rt\|_{\tilde{\dot{H}}^1_{\tau,x}}^2
	  =&2 T\lt(\tilde{\mathcal{E}}^N_D(0)\rt)^2 + 2\int_0^T \int_{-1}^1\Bigg\{ \lt(A\lt(\frac{k^N(x)}{N} \rt)-A(x)\rt) \lt(\partial_x \phi^N \rt)^2 \\
	  &- B\lt(\frac{k^N(x)}{N} \rt)\partial_x \phi^N \Bigg\}(\tau,x) dx
	\end{align*}
	Yet, as $A$ and $B$ are $\mathcal{C}^1$, we have:
	\begin{align*}
	  &\lim_{N\rightarrow + \infty} \sup_{x \in [-1,1]}\lt|A\lt(\frac{k^N(x)}{N} \rt)-A(x)\rt|=0,\\
	  &B\lt(\frac{k^N(x)}{N} \rt) \rightarrow B(x) \text{ in } L^2_x.
	\end{align*}
	Moreover:
	\begin{equation*}
	  \tilde{\mathcal{E}}_D^N(0) \underset{N\rightarrow\infty}{\rightarrow} \int_{-1}^1 \lt\{\frac{A(x)}{2} \lt(\phi^x_0(x)\rt)^2+B(x) \phi^x_0(x)+\frac{1}{2}\lt(\phi^\tau_0(x) \rt)^2 \rt\}dx.
	\end{equation*}
	Therefore, we have the following convergence:
	\begin{align}
	  \nonumber
	  \lt\|\phi^N\rt\|_{\tilde{\dot{H}}^1_{\tau,x}}^2
	  \underset{N\rightarrow\infty}{\rightarrow}&2 T\int_{-1}^1 \lt\{\frac{A(x)}{2} \lt(\phi^x_0(x)\rt)^2+B(x) \phi^x_0(x)+\frac{1}{2}\lt(\phi^\tau_0(x) \rt)^2 \rt\}dx\\
	  &- 2\int_0^T \int_{-1}^1\Bigg\{B\lt(x\rt)\partial_x \phi \Bigg\}(\tau,x) dx.
	  \label{ConvLinPreuve1}
	\end{align}
	But, from Theorem \ref{ExistUneLin}, the continuous energy $\mathcal{E}_C$ is also preserved. This implies:
	\begin{align}
	  \nonumber
	  &2 T\int_{-1}^1 \lt\{\frac{A(x)}{2} \lt(\phi^x_0(x)\rt)^2+B(x) \phi^x_0(x)+\frac{1}{2}\lt(\phi^\tau_0(x) \rt)^2 \rt\}dx\\
	  \nonumber
	  &- 2\int_0^T \int_{-1}^1\Bigg\{B\lt(x \rt)\partial_x \phi \Bigg\}(\tau,x) dx\\
	  &=\int_0^T \int_{-1}^1 \lt\{\frac{A(x)}{2} \lt(\partial_x \phi\rt)^2+\frac{1}{2}\lt(\partial_\tau \phi\rt)^2 \rt\}(\tau,x)dx.
	  \label{ConvLinPreuve2}
	\end{align}
	From \eqref{ConvLinPreuve1} and \eqref{ConvLinPreuve2}, we obtain:
	\begin{equation*}
	  \lt\|\phi^N\rt\|_{\tilde{\dot{H}}^1_{\tau,x}} \underset{N \rightarrow \infty}{\rightarrow} \lt\|\phi\rt\|_{\tilde{\dot{H}}^1_{\tau,x}}.
	\end{equation*}
	Whence $\phi^N$ strongly converges to $\phi$ in $\tilde{H}^1_{\tau,x}$ (and as a consequence, in $H^1_{\tau,x}$).
      \end{proof}

\section{The non-linear case}\label{SecNLin}
    
     This section is devoted to the proof of Theorem \ref{ThmcvNLin}.\\
     
     \begin{remark}[Boundedness of $\partial_x \phi$]
       Under the hypotheses of Theorem \ref{ThmcvNLin}, if $\partial_x \phi$ does not belong to $[a,b]$ on a non-zero measure set, then $\partial_x \phi^N$ cannot converge weakly to $\partial_x \phi$. Therefore, we henceforth assume that $\partial_x \phi(\tau,x) \in [a,b]$, $\forall x \in [-1,1], \forall \tau < T_0$.  The latter assumption holds for some $a$ and $b$ related to the initial conditions $\phi_x^0$, $\phi_\tau^0 \in \mathcal{C}^1_{p,x}$.
     \end{remark}

    \subsection{Light cone}
     
      The system \eqref{Newton} has the property that perturbations propagate at a finite speed on the macroscopic level. This is stated in Theorem \ref{ThLum}, but before proving it, we have to derive a Grönwall-type estimate:
      
      \begin{proposition}\label{proplightcone}
        Under the hypotheses of Theorem \ref{ThLum},
        if, for fixed $N$:
        \begin{equation}
          M= \lt\|U_0(.) - \tilde{U}_0(.) \rt\|_{L_t^\infty} < +\infty, \label{HypCone2}
        \end{equation} 
        then
        we have, for all $j \in \[1,N-1\]$, $t \in \R_+$:
        \begin{align}
	  \label{ConcluCone1}
          &\lt|U_j(t)-\tilde{U}_j(t) \rt| \leq M \frac{\lt(2t\sqrt{K}\rt)^{2j}}{(2j)!} \exp\lt(2t \sqrt{K} \rt),\\
          \label{ConcluCone2}
          &\lt|V_j(t)-\tilde{V}_j(t) \rt| \leq M\sqrt{2K} \lt[ \frac{\lt(2t\sqrt{K}\rt)^{2j+1}}{(2j+1)!} \exp\lt(2t \sqrt{K} \rt) + \frac{\lt(2t\sqrt{K}\rt)^{2j-1}}{(2j-1)!} \exp\lt(2t \sqrt{K} \rt) \rt].
        \end{align} 
      \end{proposition}
	
      \begin{proof}
	Remark first that it is straightforward to get \eqref{ConcluCone2} from \eqref{ConcluCone1} (\eqref{ConcluCone1} also holds for $j=0$) by integrating \eqref{Newton}. Indeed, using \eqref{Newton}, we have, for $j \in \[1,N-2\]$:
	\begin{align*}
	  \lt|V_j(t)-\tilde{V}_j(t) \rt| 
	  &\leq \int_0^t \lt| W'\lt(U_j(s)\rt)-W'\lt(\tilde{U}_{j}(s)\rt)+ W'\lt(\tilde{U}_{j-1}(s)\rt)-W'\lt(U_{j-1}(s)\rt) \rt| ds\\
	  &\overset{\eqref{HypCone1}}{\leq} K \int_0^t \lt\{ \lt|\tilde{U}_j(s)-U_j(s) \rt| + \lt|\tilde{U}_{j-1}(s)-U_{j-1}(s) \rt| \rt\} ds\\
	  &\overset{\eqref{ConcluCone1}}{\leq} 2MK \lt[\frac{t\lt(2t\sqrt{K}\rt)^{2j}}{(2j+1)!} \exp\lt(2t \sqrt{K}\rt)+\frac{t\lt(2t\sqrt{K}\rt)^{2(j-1)}}{(2j-1)!} \exp\lt(2t \sqrt{K} \rt) \rt].
	\end{align*}
	We now prove the estimate \eqref{ConcluCone1}. We do it by induction on $j$ in the expression:
	\begin{equation*}
	  S_j(t):=\max_{k \in \[j,N-1\]} \lt|U_k(t)-\tilde{U}_k(t)\rt|.
	\end{equation*}
	Using \eqref{Newton}, we get, for $j\geq 1$:
	\begin{align*}
	  \lt|U_j(t)-\tilde{U}_j(t) \rt| 
	  \leq& \int_0^t \lt|V_{j+1}(s)-\tilde{V}_{j+1}(s)+\tilde{V}_j(s)-V_j(s) \rt|ds\\
	  \leq& K\int_0^t \int_0^s \Big\{\lt|U_{j+1} - \tilde{U}_{j+1} \rt|+2 \lt|U_j-\tilde{U}_j \rt|\\
	  &+ \lt|U_{j-1}-\tilde{U}_{j-1}\rt| \Big\}(r)dr ds,\\
	  \leq & 4K \int_0^t \int_0^s S_{j-1}(r) dr ds.
	\end{align*}
	Thus:
	\begin{equation}
	  \label{Sjp1}
	  S_{j}(t) \leq 4K \int_{0}^t \int_0^s S_{j-1}(r) dr ds.
	\end{equation}
	From hypothesis \eqref{HypCone1}, we have that $S_0(0) \leq M$.  Using the same argument as above, we get for $j=0$:
	\begin{align*}
	  S_0(t) \leq M+ 4 K \int_0^t \int_0^s S_0(r)dr ds.
	\end{align*}
	Using the Grönwall Lemma, we obtain:
	\begin{equation}\label{Gron}
	  S_0(t) \leq M \exp\lt( 2t\sqrt{K}  \rt).
	\end{equation}
	Therefore, we obtain from \eqref{Sjp1} that:
	\begin{align*}
	  S_j(t) \leq & (4K)^{j} \int_0^{t=t_j} \int_0^{s_j} \int_0^{t_{j-1}} \int_0^{s_{j-1}} ... \int_0^{t_1} \int_0^{s_1} S_0(r)dr ds_1 dt_1... ds_{j-1} dt_{j-1} ds_j dt_j\\
	  \leq & (4K)^j \int_0^t \frac{(t-r)^{2j-1}}{(2j-1)!} S_0(r) dr\\
	  \leq & M \lt( 4K \rt)^j \frac{t^{2j}}{(2j)!} \exp\lt(2t \sqrt{K} \rt),
	\end{align*}
	which concludes the proof.
      \end{proof}

      We are now able to prove Theorem \ref{ThLum}:
      
      \begin{proof}
	Let $j >Nx$, $t \in [0,N\tau[$.
	By strong convexity of $W$ and by the Cauchy-Schwarz inequality, \eqref{DEC} implies:
	\begin{equation*}
	  M_N:= \sup_{t \in \R_+} \lt|U^N_0(t)\rt| \leq C\sqrt{N}.
	\end{equation*}
        By Proposition \ref{proplightcone}, we have:
        \begin{equation*}
          \lt|U_j(t)-\tilde{U}_j(t) \rt| \leq M_N \frac{\lt(2t\sqrt{K}\rt)^{2j}}{(2j)!} \exp\lt(2t \sqrt{K} \rt).
        \end{equation*}
        Using logarithm and Stirling formula, we obtain:
        \begin{align*}
	    \ln\lt(N\lt|U_j(t)-\tilde{U}_j(t) \rt|\rt) 
	    = & \ln(M_N) + 2j \ln\lt(2t\sqrt{K}\rt)-\ln((2j)!) +2t\sqrt{K}+\ln(N)\\
	    \leq & C(1+\ln(N))+ 2j\ln\lt( \frac{2et\sqrt{K}}{2j} \rt)+2t\sqrt{K},\\
	    \leq& C(1+\ln(N)) + 2j \ln\lt( \frac{\exp(2) t \sqrt{K}}{j} \rt)+(2t\sqrt{K}-2j)\\
	    \leq& C(1+\ln(N)) + 2j \ln \lt( \frac{cN\tau}{Nx} \rt)+2\lt(\frac{c}{\exp(2)}N\tau-Nx\rt)\\
	    \leq & C(1+\ln(N)) + 2N \lt( x\ln\lt(\frac{c\tau}{x} \rt)+\frac{c}{\exp(2)} \tau-x \rt)
	    .
        \end{align*}
        As $c\tau <x$, we obtain that:
        \begin{equation*}
          \lim_{N\rightarrow + \infty} \sup_{0\leq t < N\tau} \max_{j > Nx} \ln\lt(N\lt| U_j^N(t) - \tilde{U}_j^N(t) \rt|\rt) =-\infty,
        \end{equation*}
        which gives \eqref{ThLUmU}. The same method applies for \eqref{ThLumV}.
      \end{proof}

    \subsection{Strengthened convergence}
    
      In this section, we prove that if $\phi^N$ converges weakly to $\phi$ in $H^1_{\tau,x}$, then convergence holds in a strong sense for $\partial_x \phi^N$ to $\partial_x \phi$ in $L^2_{\tau,x}$.
    
      \begin{lemma}\label{LemRenforce}
	Under the hypotheses of Theorem \ref{ThmcvNLin}, if $T$ is $DC$-compatible, then the following implication is true:
        \begin{align*}
          \phi^{N} \rightharpoonup \phi \text{ in } H^1_{\tau,x}
          &&
          \Longrightarrow
          &&
          \partial_x \phi^{N} \rightarrow \partial_x \phi \text{ in } L^2_{\tau,x}.
        \end{align*}
      \end{lemma}

      To get it, we first prove two parallel integral identities, for the continuous system \eqref{WE} and the discrete one \eqref{Newton}.
      
      \begin{lemma}\label{IdSysCo}
        Assume that $W \in \mathcal{C}^2(\R)$ satisfies \eqref{Convex}. Let $\phi \in W^{1,\infty}_{\tau,x}$ be a weak entropy solution of \eqref{WE} the initial and boundary conditions \eqref{IDC} and \eqref{BCC}. Then, if $T$ is $C$-compatible, we have:
        \begin{align}
          \nonumber
          &\int_0^T \int_{-1}^1 (1-x) \lt( \partial_\tau \phi(\tau,x)-\phi^\tau_0(x) \rt) dx d\tau
           \\
          &=
          \label{IdC1}
          \int_0^T  \int_{-1}^1 (T-\tau) \lt( W'\lt(\partial_x \phi(\tau,x) \rt) - W'\lt(\phi^x_0(-1)\rt)\rt) dx d\tau.  
        \end{align}
      \end{lemma}

      A similar identity can be derived for the discrete system:
      \begin{lemma}\label{IdSysD}
        Let $W \in \mathcal{C}^1(\R)$ satisfy \eqref{Convex} and \eqref{SConvex}. Let $X_j(t)$ be the solution of \eqref{Newton} for the initial and boundary conditions \eqref{IDD} and \eqref{BCD}. Suppose that $T$ is $D$-compatible. Assume that:
        \begin{align}
	  \label{IdSysD1}
          &\partial_\tau \phi^N \rightharpoonup \ksi_\infty \text{ in } L^2_{\tau,x},\\
          &\partial_x \phi^N \rightharpoonup \psi_\infty \text{ in } L^2_{\tau,x}.
          \label{IdSysD2}
        \end{align}
        Assume that $\partial_x \phi^N$ is uniformly bounded in $\tau, x, N$. Let $\nu_{\tau,x}$ be the Young measure associated with the convergence of $\partial_x \phi^N$. Then, we have:
        \begin{align}
          \nonumber
          &\int_0^T \int_{-1}^1 (1-x) \lt( \ksi_\infty(\tau,x)-\phi^\tau_0(x) \rt) dx d\tau\\
          &=
          \label{IdD1}
          \int_0^T \int_{-1}^1 (T-\tau) \lt\{ \int_{\R} W'(\lambda) d\nu_{\tau,x}(\lambda) - W'(\phi^x_0(-1)) \rt\}d\tau dx.
        \end{align}

      \end{lemma}
      See Theorem 3.1 p 31 of \cite{muller} for the definition of Young measures.\\
      
      We temporarily admit Lemma \ref{IdSysCo} and \ref{IdSysD} and proceed with the proof of Lemma \ref{LemRenforce}.
      
      \begin{proof}
        Using Lemmata \ref{IdSysCo} and \ref{IdSysD}, we get that, if $\nu_{\tau,x}$ characterizes the weak convergence of $\partial_x \phi^N$ to $\partial_x \phi$:
        \begin{align*}
           \int_0^T \int_{-1}^1 (T-\tau) \lt\{W' \lt(\partial_x \phi (\tau,x) \rt)-\int_{\R} W'(\lambda) d\nu_{\tau,x}(\lambda) \rt\}dx d\tau=0.
        \end{align*}
        But, as $W'$ is strongly convex, using Theorem 1.1.8 p 47 of \cite{niculescu}, we obtain:
        \begin{equation*}
          \nu_{\tau,x}=\delta_{\partial_x \phi(\tau,x)}.
        \end{equation*}
	Then, Corollary 3.2 p 34 of \cite{muller} implies:
        \begin{equation*}
          \partial_x \phi^N \rightarrow \partial_x \phi \text{ in } L^2_{\tau,x}.
        \end{equation*}
      \end{proof}

      Next, we prove Lemma \ref{IdSysCo}:
      \begin{proof}
        As $T$ is $C$-compatible, we can take the following $g(\tau,x)=(T-\tau) (1-x)$ in \eqref{weakf}. Thus, we have:
        \begin{align*}
          &\int_0^T \int_{-1}^1 \lt\{ -(T-\tau) W'\lt(\partial_x \phi(\tau,x)\rt)+(1-x) \partial_\tau \phi(\tau,x) \rt\} dx d\tau\\
          &=  T\int_{-1}^1 (1-x) \phi^\tau_0(x) dx - 2\int_0^T W'\lt( \phi^x_0(-1)\rt) \lt(T-\tau\rt) d\tau.
        \end{align*}
        This is equivalent to \eqref{IdC1}.
      \end{proof}

      We now prove Lemma \ref{IdSysD}:
      \begin{proof}
        We sum \eqref{Newton} and get:
	\begin{align*}
	  \sum_{j=-N}^k \frac{d^2}{dt^2} X_j(t) &= \sum_{j=-N+1}^k W'(U_{j})(t)-W'(U_{j-1})(t)\\
	  &=W'(U_{k})(t)-W'(U_{-N})(t).
	\end{align*}
	Then we multiply the above expression by $(NT-t)$ and integrate with respect to $t$:
	\begin{align*}
	  &\int_0^{NT} (NT-t) \lt( W'(U_{k})-W'(U_{-N}) \rt)(t) dt \\
	  &= \int_0^{NT} \int_0^{t} \sum_{i=-N}^k \frac{d^2}{dt^2} X_j(s) ds dt\\
	  &= \sum_{j=-N}^k \int_0^{NT} \lt\{ \frac{d X_j}{dt}(N\tau)- \frac{d}{dt} X_j(0) \rt\}dt.
	\end{align*}
	If we rescale it and sum over $k \in \[-N,N-1\]$, we obtain:
	\begin{align*}
	  &\frac{1}{N} \sum_{k=-N}^{N-1}\int_0^{T} (T-\tau) \lt(W'\lt(\partial_x  \phi^N \lt(\tau,\frac{k}{N} \rt)\rt)-W'\lt(U_{-N}(N\tau) \rt) \rt) d\tau\\
	  &= \frac{1}{N^2} \sum_{k=-N}^{N-1} \sum_{j=-N}^k \int_0^{T} \lt\{ \ksi^N\lt(\tau,\frac{k}{N} \rt)- \ksi^N\lt(0,\frac{k}{N} \rt) \rt\}d\tau\\
	  &=\frac{1}{N} \sum_{k=-N}^{N-1} \frac{(N-k)}{N} \int_0^{T} \lt\{ \ksi^N\lt(\tau,\frac{k}{N} \rt)- \ksi^N \lt(0,\frac{k}{N} \rt) \rt\}d\tau\\
	  &= \int_0^T \int_{-1}^1 \lt(1-\frac{\lfloor Nx \rfloor}{N}\rt) \lt( \ksi^N\lt(\tau,x \rt)- \ksi^N \lt(0,x \rt) \rt)d\tau. 
	\end{align*}
	Remark first that:
	\begin{equation*}
	  1-\frac{\lfloor Nx \rfloor}{N} \rightarrow 1-x \text{ in } L^\infty_x.
	\end{equation*}
	As $T$ is $D$-compatible, we have the following convergence, uniform for $\tau \in [0,T]$:
	\begin{equation*}
	  W'\lt(U_{-N}(N\tau)\rt) \rightarrow W'\lt( \phi^x_0(-1)\rt).
	\end{equation*}
	From initial conditions \eqref{IDD}, we get that:
	\begin{equation*}
	  \ksi^N\lt(0,.\rt) \rightarrow \phi^\tau_0(.) \text{ in } L^2_x,
	\end{equation*}
	and from the hypotheses \eqref{IdSysD1} and \eqref{IdSysD2} that:
	\begin{align*}
          &\int_0^T \int_{-1}^1 (1-x) \lt( \ksi_\infty(\tau,x)-\phi^\tau_0(x) \rt) dx d\tau\\
          &=
          \int_0^T \int_{-1}^1 (T-\tau) \lt\{ \int_{\R} W'(\lambda) d\nu_{\tau,x}(\lambda) - W'(\phi^x_0(-1)) \rt\}d\tau dx.
        \end{align*}
      \end{proof}
    
    \subsection{From strong convergence of $\partial_x \phi^N$ to strong convergence of $\partial_\tau \phi^N$}
    
      We now prove that the strong convergence of $\partial_x \phi^N$ in $L^2_{\tau,x}$ implies the strong convergence of $\partial_\tau \phi^N$ in $L^2_{\tau,x}$.
      \begin{lemma}\label{LemAuBinLions}
        Under the hypotheses of Theorem \ref{ThmcvNLin}, if $T$ is $DC$-compatible, the following implication is true:
        \begin{align*}
          \lt\{
	    \begin{array}{l l}
	      \phi^{N} \rightarrow \phi &\text{ in } L^2_{\tau,x}\\
	      \partial_x \phi^N \rightarrow \partial_x \phi& \text{ in } L^2_{\tau,x}\\
	      \partial_\tau \phi^N \rightharpoonup \partial_\tau \phi& \text{ in } L^2_{\tau,x}
	    \end{array}
          \rt.
          &&
          \Longrightarrow
          &&
          \lt\{
          \begin{array}{l l}
            \partial_\tau \phi^N \rightarrow \partial_\tau \phi &\text{ in } L^2_{\tau,x}\\
            \ksi^N \rightarrow \partial_\tau \phi &\text{ in } L^2_{\tau,x}
          \end{array}
	  \rt.
	  ,
	\end{align*}
	for $\ksi^N(\tau,x)=V_{k^N(x)}^N(N\tau)$.
      \end{lemma}
      
      We first prove  that it is enough to have $\partial_\tau \phi^N \rightarrow \partial_\tau \phi$.
      
      \begin{lemma}\label{AubinAide}
        Assume that $\partial_\tau \phi^N \rightarrow \zeta$ in $L^2_{\tau,x}$. Then $\ksi^N \rightarrow \zeta$ in $L^2_{\tau,x}$.
      \end{lemma}
      
      \begin{proof}
        Recall that $\zeta^N=\partial_\tau \phi^N$. Assume that $\zeta^N \rightarrow \zeta$. 
        We suppose first that:
        \begin{equation}\label{formeC}
          \zeta(\tau,x)= \sum_{k=-N}^{N-1} \zeta^k(\tau) \mathbb{1}_{\lt[\frac{k}{N}, \frac{k+1}{N} \rt[}(x).
        \end{equation} 
	Then, for all $j \in \[-N,N-1\]$, $\tau \in [0,T]$:
	\begin{align*}
	  &\int_{j/N}^{(j+1)/N} \lt( \zeta^N - \zeta\rt)^2(\tau,x)dx \\
	  &=\frac{1}{N} \int_0^1 \Big( \theta(\ksi^N(\tau,j/N)-\zeta(\tau,j/N))\\
	  &+(1-\theta) (\ksi^N(\tau,(j+1)/N)-\zeta(\tau,j/N)) \Big)^2d\theta\\
	  &=\frac{1}{N} \Bigg[\frac{1}{3} \lt(\ksi^N(\tau,j/N)-\zeta(\tau,j/N)\rt)^2+\frac{1}{3}\lt(\ksi^N(\tau,(j+1)/N)-\zeta(\tau,j/N)\rt)^2 \\
	  &+\frac{1}{3}\lt(\ksi^N(\tau,j/N)-\zeta(\tau,j/N)\rt)\lt(\ksi^N(\tau,(j+1)/N)-\zeta(\tau,j/N)\rt) \Bigg]\\
	  &\geq\frac{1}{6N} \lt(\ksi^N(\tau,j/N)-\zeta(\tau,j/N)\rt)^2\\
	  &\geq \frac{1}{6}\int_{j/N}^{(j+1)/N} \lt( \ksi^N(\tau,x) -\zeta(\tau,x)\rt)^2dx.
	\end{align*}
	Therefore, we have:
	\begin{equation*}
	  \int_0^T \int_{-1}^1 \lt( \zeta^N - \zeta\rt)^2(\tau,x) dxd\tau \geq \frac{1}{6} \int_{0}^T \int_{-1}^1 \lt(\ksi^N-\zeta \rt)^2 (\tau,x) dx d\tau.
	\end{equation*}
	If $\zeta$ is  now only $L^2$, we approximate it by $\tilde{\zeta}$ that has the form \eqref{formeC}. Thus, by the triangle inequality, and applying the former result on $\tilde{\zeta}$, we have:
	\begin{align*}
	  \int_0^T \int_{-1}^1 \lt(\ksi^N-\zeta \rt)^2 (\tau,x) dx d\tau \leq&
	  C \int_0^T \int_{-1}^1 \lt\{\lt(\zeta^N-\zeta \rt)^2 (\tau,x) + \lt(\zeta-\tilde{\zeta} \rt)^2 (\tau,x)\rt\} dx d\tau.
	\end{align*}
	Hence:
	\begin{equation*}
	  \lt\| \ksi^N - \zeta \rt\|_{L^2_{\tau,x}} \rightarrow 0.
	\end{equation*}
      \end{proof}

      We can now proceed with the proof of Lemma \ref{LemAuBinLions}.
      \begin{proof}
        By definition:
        \begin{align}\label{Aubin0}
          & \zeta^N(\tau,x)=\lt( \lt(1-\theta^N(x) \rt)\frac{d}{dt} X_{k^N(x) }+\theta^N(x)\frac{d}{dt}X_{k^N(x) +1} \rt)(t).
        \end{align}
	Using \eqref{Newton}, we obtain:
	\begin{align*}
	  \partial_\tau \zeta^N(\tau,x)
	  =& N\lt(1-\theta^N(x)\rt) \lt(W'\lt(U_{k(x)}\rt)-W'\lt(U_{k(x)-1} \rt) \rt)(t)\\
	  &+N\theta^N(x)\lt(W'\lt(U_{k(x)+1}\rt)-W'\lt(U_{k(x)} \rt) \rt)(t).
	\end{align*}
	We define:
	\begin{equation*}
	  \Xi(x):=
	    \lt\{
	  \begin{array}{l c l}
	    0 &\text{ if } & x <0, \\
	    \frac{x^2}{2} & \text{ if } & x \in [0,1],\\
	    \frac{1}{2}-(x-1)^2+(x-1) & \text{ if } & x \in [1,2],\\
	    \frac{1}{2}+\frac{(x-2)^2}{2}-(x-2) & \text{ if } & x \in [2,3],\\
	    0 &\text{ if } & x>3. 
	  \end{array}
	  \rt.
	\end{equation*}
	Let $\Xi^N_j(x):=\Xi(Nx-j)$. We have:
	\begin{equation*}
	  \frac{d}{dx}\Xi^N_j(x):=
	    \lt\{
	  \begin{array}{l c l}
	    0 &\text{ if } & x <j/N, \\
	    N\theta^N(x) & \text{ if } & x \in [j/N,(j+1)/N],\\
	    N\lt(1-2\theta^N(x)\rt) & \text{ if } & x \in [(j+1)/N,(j+2)/N],\\
	    N\lt( \theta^N(x) -1 \rt)& \text{ if } & x \in [(j+2)/N,(j+3)/N],\\
	    0 &\text{ if } & x>(j+3)/N. 
	  \end{array}
	  \rt.
	\end{equation*}
	Newt, we define:
	\begin{equation*}
	  \Psi(\tau,x):=\sum_{j=-N}^{N-1} \Xi^N_{j-1}(x) W'\lt(U_j(N\tau)\rt).
	\end{equation*}
	We have:
	\begin{align*}
	  \partial_x \Psi^N(\tau,x)=\partial_\tau\zeta^N(\tau,x), &&\forall \tau \in [0,T], \forall x \in [-1,1].
	\end{align*}
	Since $\partial_x \phi^N \rightarrow \partial_x \phi$ in $L^2_{\tau,x}$, and since $\partial_x \phi^N$ is bounded in $L^\infty$, we have:
	\begin{equation*}
	  W'\lt(\partial_x \phi^N\rt) \rightarrow W'\lt(\partial_x \phi\rt) \text{ in } L^2_{\tau,x}.
	\end{equation*}
	We claim that it implies:
	\begin{equation}\label{L2Psi}
	  \lt\| \Psi^N-W'\lt(\partial_x \phi \rt) \rt\|_{L^2_{\tau,x}}  \underset{N \rightarrow + \infty}{\rightarrow} 0.
	\end{equation}
	Indeed, as $\partial_x \phi^N$ is bounded in $L^\infty_{\tau,x}$ uniformly in $N$, it suffices to bound as follows:
	\begin{align*}
	  &\int_{-1}^1 \lt| \Psi^N(\tau,x)-W'\lt(\partial_x \phi(\tau,x)\rt) \rt| dx \\
	  &\leq \int_{-1}^1 \lt| \sum_{j=-N}^{N-1} \Xi^{N}_{j-1}(x) \lt(W'\lt(U_j(N\tau) \rt)-W'\lt(\partial_x \phi(\tau,x) \rt) \rt) \rt| dx\\
	  &\leq \int_{-1}^1  \sum_{j=-N}^{N-1} \Xi^{N}_{j-1}(x) \lt|W'\lt(U_j(N\tau) \rt)-W'\lt(\partial_x \phi(\tau,x) \rt) \rt| dx\\
	  & \leq \int_{-1}^1 \sum_{j=-N}^{N-1} \mathbb{1}_{\lt[\frac{j-1}{N},\frac{j+2}{N}\rt]}(x) \lt|W'\lt(U_j(N\tau) \rt)-W'\lt(\partial_x \phi(\tau,x) \rt) \rt| dx\\
	  & \leq  \sum_{j=-1}^1 \int_{-1}^1 \lt|W'\lt(U_{k(x)+j}(N\tau)  \rt)-W'\lt(\partial_x \phi(\tau,x) \rt) \rt| dx \\
	  &\leq \sum_{j=-1}^1 \int_{-1}^1 \lt|W'\lt( \partial_x \phi^N\lt(\tau,x+\frac{j}{N} \rt) \rt)-W'\lt( \partial_x \phi\lt(\tau,x \rt)\rt) \rt| dx\\
	  & \rightarrow 0.
	\end{align*}
	Interpolating with $L^\infty_{\tau,x}$, this gives \eqref{L2Psi}. We define:
	\begin{align*}
	  \alpha^N :=\zeta^N-\zeta,&&
	  \beta^N:=\Psi- W'\lt(\partial_x \phi\rt),&&
	  \gamma^N:=\partial_x \phi^N - \partial_x \phi.
	\end{align*}
	Remark that, by definition, we have:
	\begin{align}
	  \label{Aubin21}&\partial_\tau \alpha^N = \partial_x  \beta^N,\\
	  \label{Aubin22}&\partial_x \alpha^N = \partial_\tau \gamma^N.
	\end{align}
	$D$-compatibility of $T$ implies the following convergences:
	\begin{align}
	  \label{CondBordAubin}
	   &\int_0^T \lt(\alpha^N(\tau,-1)\rt)^2 \rightarrow 0, && \int_0^T\lt(\alpha^N(\tau,1)\rt)^2 d\tau \rightarrow 0, ~~~
	   \int_{-1}^1 \lt(\alpha^N(0,x) \rt)^2 dx \rightarrow 0.
	\end{align} 
	From \eqref{L2Psi}, and by $D$-compatibility, we deduce that:
	\begin{align}
	  &\lim_{N\rightarrow +\infty} \lt\| \beta^N \rt\|_{L^2_{\tau,x}} =0,
	  &&\lim_{N \rightarrow +\infty} \lt\{\lt\| \beta^N(.,1) \rt\|_{L^2_\tau} + \lt\| \beta^N(.,-1) \rt\|_{L^2_\tau} \rt\}  =0. \label{borbeta}
	\end{align}
	The energy estimates for \eqref{Newton} and \eqref{WE} give:
	\begin{align}
	  &\sup_{N} \lt\| \gamma^N \rt\|_{L^2_{\tau,x}} \leq C,
	  && \sup_{N}\lt\{ \lt\|\gamma^N(0,.) \rt\|_{L^2_x} +\lt\|\gamma^N(T,.) \rt\|_{L^2_x}\rt\} \leq C\label{borgamma}.
	 \end{align}
	We now claim that  \eqref{Aubin21}, \eqref{Aubin22}, \eqref{borbeta}, \eqref{borgamma},  \eqref{CondBordAubin} imply:
	\begin{equation}\label{ClaimAubin}
	  \limsup_{N\rightarrow + \infty}\lt\| \alpha^N \rt\|_{L^2_{\tau,x}} = 0,
	\end{equation}
	which gives the desired result, thanks to Lemma \ref{AubinAide}. To prove \eqref{ClaimAubin}, we write:
	\begin{align*}
	  &\alpha^N(\tau,x)\overset{\eqref{Aubin21}}{=}\alpha^N(\tau,-1) + \int_{-1}^x \partial_\tau \gamma^N(\tau,y) dy.\\
	  &\alpha^N(\tau,x)\overset{\eqref{Aubin22}}{=}\alpha^N(0,x) + \int_0^\tau \partial_x \beta^N(\nu,x)d\nu.
	\end{align*}
	Therefore:
	\begin{align*}
	  &\int_0^T \int_{-1}^1 \lt(\alpha^N(\tau,x)\rt)^2 dx d\tau\\
	  &= \int_0^T \int_{-1}^1 
	  \lt(\alpha^N(\tau,-1) + \int_{-1}^x \partial_\tau \gamma^N(\tau,y) dy \rt)
	  \lt( \alpha^N(0,x) + \int_0^\tau \partial_x \beta^N(\nu,x)d\nu\rt)
	  dx d\tau\\  
	  &= \int_0^T \int_{-1}^1 \alpha^N(\tau,-1) \alpha^N(0,x) dxd\tau 
	  +\int_ 0^T \int_{-1}^1 \alpha^N(\tau,-1) \int_0^\tau \partial_x \beta^N(\nu,x) d\nu dx d\tau\\
	  &~~~+ \int_0^T \int_{-1}^1 \alpha^N(0,x) \int_{-1}^x \partial_\tau \gamma^N(\tau,y) dy dx d\tau \\
	  &~~~+ \int_0^T \int_{-1}^1 \int_{-1}^x \partial_\tau \gamma^N(\tau,y)dy \int_0^\tau \partial_x \beta^N(\nu,x) d\nu dx d\tau\\
	  &=:T^N_1+T^N_2+T^N_3+T^N_4.
	\end{align*}
	We deal separately with $T^N_1, T^N_2, T^N_3, T^N_4$.
	By the Cauchy-Schwarz inequality:
	\begin{align}\label{AubinT1}
	  T^N_1 \leq \lt( \int_0^T \alpha^N(\tau,-1) d\tau\rt)^{1/2}\lt( \int_{-1}^1 \alpha^N(0,x) dx\rt)^{1/2} \overset{\eqref{CondBordAubin}}{\rightarrow} 0.
	\end{align}
	Integrating over $x$ in $T_2^N$, we obtain:
	\begin{align}\label{AubinT2}
	  T_2^N =& \int_0^T \alpha^N(\tau,-1) \int_0^\tau \lt(\beta^N(\nu,1)-\beta^N(\nu,-1) \rt) d\nu d\tau \overset{\eqref{borbeta},\eqref{CondBordAubin}}{\rightarrow} 0.
	\end{align}
	Next, integrating over $\tau$ in $T_3^N$, we get:
	\begin{align}\label{AubinT3}
	  T_3^N =& \int_{-1}^1 \alpha^N(0,x)\int_{-1}^x \lt(\gamma^N(T,y)-\gamma^N(0,y) \rt)dy dx \overset{\eqref{borgamma}, \eqref{CondBordAubin}}{\rightarrow} 0.
	\end{align}
	We deal with $T_4^N$ by a double integration by parts:
	\begin{align}
	  \nonumber
	  T_4^N=& \int_0^T \int_{-1}^1 \partial_\tau \gamma^N(\tau,y)dy \int_0^\tau \beta^N(\nu,1) d\nu d\tau
	  \\
	  \nonumber
	  &- \int_0^T \int_{-1}^1 \partial_\tau \gamma^N(\tau,x) \int_0^\tau \beta^N(\nu,x) d\nu dx d\tau
	  \\
	  \nonumber
	  =& \int_{-1}^1 \gamma^N(T,y) dy \int_0^T \beta^N(\tau,1) d\tau
	  \\
	  \nonumber
	  &-\int_0^T\int_{-1}^1 \gamma^N(\tau,x) \beta^N(\tau,1) dx d\tau
	  \\
	  \nonumber
	  &-\int_0^T \int_{-1}^1 \beta^N(\tau,x) \gamma^N(T,x) d\tau dx  \\
	  \nonumber
	  & \int_0^T \int_{-1}^1 \beta^N(\tau,x) \gamma^N(\tau,x) d\tau dx.
	  \\
	  &\overset{\eqref{borbeta},\eqref{borgamma}}{\rightarrow} 0.\label{AubinT4}
	\end{align}
	From \eqref{AubinT1}, \eqref{AubinT2}, \eqref{AubinT3}, \eqref{AubinT4}, we obtain \eqref{ClaimAubin}, which concludes the proof.
      \end{proof}

    \subsection{Proof of Theorem \ref{ThmcvNLin}}
    
      We are now in position to prove Theorem \ref{ThmcvNLin}.
      \begin{proof}
        We first prove the existence of a $D$-compatible $T>0$. $\phi^x_0$ and $\phi^\tau_0$ satisfies \eqref{Riemannbord}, and : \begin{equation*}
	  \lt\|W''(U_j(t))\rt\|_{L^\infty_t(l^\infty_j)} \leq \lt\|W''\rt\|_{L^\infty([a,b])}.
	\end{equation*}
	Moreover, the discrete energy is preserved, which implies that:
	\begin{align*}
	  \mathcal{E}_D(t) \leq& \sum_{j=-N+1}^N \lt\{ \frac{1}{2}\lt( \phi^\tau_0\lt(\frac{j}{N} \rt)\rt)^2 + W\lt(\phi^x_0 \lt(\frac{j}{N} \rt) \rt)  \rt\}\\
	  \leq & N \lt\| \phi_\tau\rt\|_{L^\infty_x}^2 + 2N \lt\|W\rt\|_{L^\infty([a,b])}. 
	\end{align*}
	Therefore, we can apply Theorem \ref{ThLum}. Let $\tilde{X}_j(t)$ satisfy \eqref{Newton} with initial conditions $\tilde{U}_j(t=0)=u_r$, $\tilde{V}_j(t=0)=v_r$ and Dirichlet boundary conditions (remark that this means that $U_j(t)$ and $V_j(t)$ do not depend on time). We compare $\tilde{U}_j(t)=u_r$ and $\tilde{V}_j(t)=v_r$ with $U_j(t)$ and $V_j(t)$,  respectively. Using Theorem \ref{ThLum}, there exists $c>0$ such that, if $T<1/(4c)$ and $T<T_0$:
	\begin{align*}
	  &\sup_{t \in [0,NT]}\max_{j \in [3N/4,N]} N\lt|U_j(t)-u_r\rt| \underset{N \rightarrow +\infty}{\rightarrow} 0,\\
	  &\sup_{t \in [0,NT]}\max_{j \in [3N/4,N]} N\lt|V_j(t)-v_r\rt| \underset{N \rightarrow +\infty}{\rightarrow} 0.
	\end{align*}
	The same argument applies for $j<-3N/4$. Therefore, there exists a $D$-compatible $T>0$.
	
	We now prove that $\phi^N$ does not converge to $\phi$. We argue by contradiction and assume that:
	\begin{equation}\label{abs1}
	  \phi^N \rightarrow \phi \text{ in } D_{\tau,x}'.
	\end{equation}
	Since the discrete energy $\mathcal{E}_D$ is preserved and as $W$ is strongly convex, we get an $H^1$ estimate over $\phi^N$:
	\begin{equation*}
	  \int_{-1}^1 \lt\{ \lt(\partial_x \phi(\tau,x)\rt)^2 + \lt( \partial_\tau \phi(\tau,x)\rt)^2\rt\} \leq C. 
	\end{equation*}
	which directly implies:
	\begin{equation}\label{abs2}
	  \phi^N \rightharpoonup \phi \text{ in } H^1_{\tau,x}.
	\end{equation}
	By Lemma \ref{LemRenforce}, we get that:
	\begin{equation}\label{abs3}
	  \partial_x \phi^N \rightarrow \partial_x \phi \text{ in } L^2_{\tau,x}.
	\end{equation}
	Whence, by Lemme \ref{LemAuBinLions}, we have:
	\begin{equation}\label{abs4}
	  \ksi^N \rightarrow \partial_\tau \phi \text{ in } L^2_{\tau,x}.
	\end{equation}
	$W$ is continuous. Therefore \eqref{abs3}, \eqref{BorneLinfty} and \eqref{abs4} imply:
	\begin{align}
	   \nonumber
	   &\int_0^T \int_{-1}^1 \lt\{ \frac{1}{2} \lt(\partial_\tau \phi^N \rt)^2+ W\lt( \partial_x \phi^N \rt) \rt\}(\tau,x) dx d\tau\\
	   &\rightarrow \int_0^T \int_{-1}^1 \lt\{ \frac{1}{2} \lt(\partial_\tau \phi \rt)^2+ W\lt( \partial_x \phi \rt) \rt\}(\tau,x) dx d\tau.\label{abs5}
	\end{align}
	But the left-hand term of \eqref{abs5} also converges, by discrete energy conservation, to:
	\begin{align*}
	  &\int_0^T \int_{-1}^1 \lt\{ \frac{1}{2} \lt(\partial_\tau \phi^N \rt)^2+ W\lt( \partial_x \phi^N \rt) \rt\}(\tau,x) dx d\tau\\
	  &\rightarrow \int_0^T \int_{-1}^1 \lt\{ \frac{1}{2} \lt(\phi^\tau_0 \rt)^2+ W\lt( \phi^x_0 \rt) \rt\}(\tau,x) dx d\tau=T\mathcal{E}_C(0),
	\end{align*}
	and the right-hand term of \eqref{abs5} is nothing but the energy $\mathcal{E}_C(\tau)$. As $\phi$ is an entropy solution, we have:
	\begin{align*}
	  \int_0^T \mathcal{E}_C(\tau) d\tau =& \int_0^{T_1} \mathcal{E}_C(\tau) d\tau + \int_{T_1}^T \mathcal{E}_C(\tau) d\tau\\
	  \overset{\eqref{PertEc}}{<}& T \mathcal{E}_C(0).
	\end{align*}
	Therefore, we reach a contradiction, and $\phi^N$ cannot converge to $\phi$.
      \end{proof}

  \section{A uniform bound on the distance between particles}\label{SecConjec}
    
    Notice that it is important to assume some regularity on the initial conditions in Conjecture \ref{ConjecDur}. It is indeed possible to build some initial conditions that are small in $l^\infty_j$ such that the associated solutions of \eqref{Newton} are not bounded uniformly in $N$ at a fixed macroscopic time $\tau>0$.
    The following proof of Proposition \ref{ConjSubtl} uses the reversibility of equation \eqref{Newton} and also a linearization of eigenvalues $\lambda_k$.

     We first derive explicit formulae for solution of linear periodic system \eqref{Newton}.\\
      Let $I \in \mathcal{M}_{2N}(\R)$ be the identity matrix, and $J \in \mathcal{M}_{2N}(\R)$ the circular permutation:
      \begin{equation*}
        J_{jk}:= \delta^{j}_{k+1}.
      \end{equation*}
      When the potential is quadratic and satisfies \eqref{PotQuad}, system \eqref{Newton} with periodic boundary conditions is equivalent to:
      \begin{equation*}
        \frac{d^2X}{dt^2} + \lt( 2 I - J - J^{-1} \rt)X=0.
      \end{equation*}
      We diagonalize this system using its eigenvectors:
      \begin{equation*}
	\Omega_j=
	\frac{1}{\sqrt{2N}}\lt(
	  \begin{array}{l}
		  1\\
		  \omega_j\\
		  ..\\
		  \omega_j^{2N-1}
	  \end{array}
	\rt),
      \end{equation*}
      where $\omega_j=\exp\lt({\frac{ij\pi}{N}}\rt)$. The associated eigenvalues are:
      \begin{equation}\label{Lambda}
	\lambda_j = 2\lt(1-\cos\lt(\frac{j\pi}{N}\rt)\rt)=4 \sin^2\lt(\frac{j\pi}{2N} \rt).
      \end{equation}
      Thus, a solution of \eqref{Newton} with periodic boundary conditions satisfies:
      \begin{align}
	X_j(t)=
	&\sum_{k=1}^{2N-1} \cos\lt(t\sqrt{\lambda_k}\rt) \lt(\Omega_k | {X}(0) \rt) \lt(\Omega_k\rt)_j+\sum_{k=1}^{2N-1} \frac{1}{\sqrt{\lambda_k}}\sin\lt(t\sqrt{\lambda_k}\rt) \lt(\Omega_k | {V}(0) \rt) \lt(\Omega_k\rt)_j \nonumber\\
	&+ \frac{1}{\sqrt{2N}} \lt[\lt(\Omega_0|{X}(0)\rt) + \lt(\Omega_0|{V}(0)\rt) t \rt] \label{ExprElle},
      \end{align}
      where, for two vectors $Y$, $Z \in \mathbb{C}^{2N}$, $(|)$ denotes the hermitian product:
      \begin{equation*}
        \lt(Y|Z\rt):=\sum_{k=-N}^{N-1} Y_k Z_k^*.
      \end{equation*}
      One easily derives such formulae for $V_j$, $U_j$ and $Z_j$ by linearity.

      We can now prove Proposition \ref{ConjSubtl}.
      \begin{proof}[Proof of Proposition \ref{ConjSubtl}]
	Using the reversibility of \eqref{Newton}, it is enough to prove that, if $X^N_j$ is a solution of \eqref{Newton} with periodic boundary condition such that
	$U^N_j(t=0)=\delta_j^0$ and $V^N_j(t=0)=0$, then:
	\begin{align}
	&\lt\|U^N_j(N\tau)\rt\|_{l_j^\infty} \rightarrow 0,\label{Amontrer1}\\
	&\lt\|V^N_j(N\tau)\rt\|_{l_j^\infty} \rightarrow 0.\label{Amontrer2}
	\end{align}
	Indeed, let $\tilde{X}_i$ be the solution of \eqref{Newton} with periodic boundary conditions and the following initial conditions:
	\begin{align*}
	&\tilde{U}^N_j(t=0):=K_N U^N_j(N\tau),\\
	&\tilde{V}^N_j(t=0):=-K_N V^N_j(N\tau).
	\end{align*}
	By linearity and reversibility of \eqref{Newton}, we get:
	\begin{align*}
	&\tilde{U}^N_j(N\tau)=K_N U^N_j(0),\\
	&\tilde{V}^N_j(N\tau)=-K_N V^N_j(0).
	\end{align*}
	Setting $K_N=\lt\{\max\lt(\lt\|U^N_j(N\tau)\rt\|_{l_j^\infty},\lt\|V^N_j(N\tau)\rt\|_{l_j^\infty} \rt)\rt\}^{-1}$ gives the desired result.
	
	We only show \eqref{Amontrer1}, as the proof of \eqref{Amontrer2} is similar. Thanks to \eqref{ExprElle}:
	\begin{align*}
	U_k^N(N\tau)=&\frac{1}{2N}\sum_{j=-N}^{N-1} \cos\lt(2N\tau\sin\lt(\frac{j\pi}{2N}\rt)\rt) e^{ijk\pi/N}.
	\end{align*}
	To simplify the proof, we suppose $N=nm$ (it can be generalized with a few technicalities). Thus:
	\begin{align*}
	U_k^N(N\tau)=&\frac{1}{2N}\sum_{l=-n}^{n-1} \sum_{j=0}^{m-1} \cos\lt(2N\tau\sin\lt(\frac{(lm+j)\pi}{2N}\rt)\rt) e^{i(lm+j)k\pi/N}.
	\end{align*}
	Let us bound terms of the type:
	\begin{align*}
	Q_{\pm}^{l,k}:=\lt|\frac{1}{N}\sum_{j=0}^{m-1} \exp\lt(\pm i2N\tau\sin\lt(\frac{(lm+j)\pi}{2N}\rt)\rt) e^{i(lm+j)k\pi/N}\rt|.
	\end{align*}
	We expand:
	\begin{equation*}
	\sin\lt(\frac{(lm+j)\pi}{2N}\rt)=\sin\lt(\frac{lm\pi}{2N}\rt)+\frac{j\pi}{2N}\cos\lt(\frac{lm\pi}{2N} \rt)+\frac{m^2}{N^2}r(N,m,l,j),
	\end{equation*}
	where $r(N,m,l,j)<C$, independently of $N$, $m$, $l$, $j$. As a consequence:
	\begin{align*}
	Q_{\pm}^{l,k} 
	\leq& \lt|\frac{1}{N}\sum_{j=0}^{m-1} \exp\lt(\pm 2i \lt( \frac{\tau j\pi}{2}\cos\lt(\frac{lm\pi}{2N} \rt)+\frac{m^2\tau}{N}r(N,m,l,j)\rt)\rt) e^{ijk\pi/N}\rt|\\
	\leq& \frac{1}{N} \sum_{j=0}^{m-1} \frac{Cm^2}{N}
	+\lt|\frac{1}{N}\sum_{j=0}^{m-1} \exp\lt(\pm   i\tau j\pi\cos\lt(\frac{lm\pi}{2N}\rt)+ \frac{ijk\pi}{N}\rt)\rt|\\
	\leq& C \frac{m^3}{N^2}+\frac{1}{N} \frac{2}{\lt|1-\exp(i\gamma_{\pm}(k,l,m,N,\tau))\rt|},
	\end{align*}
	where:
	\begin{equation*}
	\gamma_\pm(k,l,m,N,\tau)=\frac{k\pi}{N}\pm \tau \pi \cos\lt(\frac{lm\pi}{2N} \rt).
	\end{equation*}
	Without loss of generality, we focus only on $Q_+^{l,k}$. There exist at most two solutions $s_1$ and $s_2 \in [-\pi/2,\pi/2]$ to the equation:
	\begin{equation*}
	  \tau\pi\cos(s)+\frac{k\pi}{N}=0.
	\end{equation*}
	Let $1>\delta>0$. If $\frac{lm\pi}{2N} \notin ]s_1-\delta,s_1+\delta[ \cup ]s_2-\delta,s_2+\delta[ $, then:
	\begin{equation}
	\lt|1-\exp(i\gamma_{+}(k,l,m,N,\tau))\rt|>C\delta^2.
	\end{equation}
	for some universal constant $C$. Whence, if $\frac{lm\pi}{2N} \notin ]s_1-\delta,s_1+\delta[ \cup ]s_2-\delta,s_2+\delta[ $, we have:
	\begin{align*}
	  Q_+^{l,k}\leq& C \frac{m^3}{N^2}+\frac{1}{N} \frac{C}{\delta^2}.
	\end{align*}
	Moreover, it is immediate from the definition of $Q^{l,k}_{\pm}$ that for all $l, k$:
	\begin{equation*}
	Q_{\pm}^{l,k} \leq \frac{m}{N}.
	\end{equation*}
	We denote:
	\begin{equation*}
	E:= \[-n,n-1\] \cap \lt( \lt] \frac{2N}{m\pi}(s_1-\delta), \frac{2N}{m\pi}(s_1+\delta)\rt[ \cup \lt]\frac{2N}{m\pi}(s_2-\delta),\frac{2N}{m\pi}(s_2+\delta) \rt[\rt),
	\end{equation*}
	and bound the sum:
	\begin{align*}
	\sum_{l=-n}^{n-1} Q_+^{l,k} 
	=& \sum_{l\in E} Q^{l,k}_+ + \sum_{l \notin E} Q^{l,k}_+\\
	\leq& \sum_{l \in E}\frac{m}{N} + C\sum_{l \notin E} \frac{m^3}{N^2} + \frac{1}{N\delta^2}\\
	\leq& \frac{8N}{m\pi} \delta \frac{m}{N} +2Cn \lt(\frac{m^3}{N^2} + \frac{1}{N\delta^2}\rt)\\
	\leq& C \lt( \delta + \frac{m^2}{N} + \frac{1}{m \delta^2} \rt).
	\end{align*}
	Let $\delta=\lt( m\rt)^{-1/3}$, $m=N^{3/7}$. We get:
	\begin{align*}
	\sum_{l=-n}^{n-1} Q_+^{l,k} \leq \frac{C}{N^{1/7}}.
	\end{align*}
	Doing the same manipulations on $Q_-^{l,k}$, we get that, for all $k \in \[-N,N-1\]$:
	\begin{align*}
	\lt|U_k^N(N\tau)\rt| &\leq \sum_{l=-n}^{n-1} Q_+^{l,k} + Q_-^{l,k}\\
	&\leq \frac{C}{N^{1/7}},
	\end{align*}
	whence \eqref{Amontrer1}.
      \end{proof}
      
      \begin{remark}
	One can remove the technical assumption $N=nm$ with $m, n \in \mathbb{N}$, by fixing $\mu:=\lt \lfloor N^{1/7} \rt\rfloor$, and then $m:=\mu^3$, $n:=\mu^4$. Remarking that $N-\mu^7 <C  N^{6/7}$, we can apply the same proof as above and derive the same estimates.
      \end{remark}
    
\section{Non-existence of discrete shock waves}\label{SecSoliton}
    
    We prove in this section that there do not exist discrete shock waves. We use some ideas from \cite{Benzoni}, where an existence result is proven for upwind schemes.

    \subsection{Quadratic potential}
    
      In this section, we prove Proposition \ref{DiscreteSolitonsLin}.
      We first show a lemma which is valid for a wide class of potentials $W$:
      \begin{lemma}\label{StatWave}
        Suppose $W \in \mathcal{C}^1(\R)$, such that $W'(u_l) \neq W'(u_r)$. Then there exists no discrete shock wave to the equation \eqref{Newton} with zero speed. That is, there does not exist $X_j(t)$ satisfying Definition \ref{DiscShock}, with associated $c=0$.
      \end{lemma}
      
      \begin{proof}
        Integrating \eqref{Soliton}, we get:
        \begin{align*}
          c^2 \int_x^y \phi''(s) ds = \int_{y-1}^y W'(\phi(s+1)-\phi(s)) ds - \int_{x-1}^x W'(\phi(s+1)-\phi(s)) ds.
        \end{align*}
        If $x \rightarrow -\infty$ and $y \rightarrow +\infty$, we obtain:
        \begin{equation}
          \label{RH2}
          c^2 (u_r-u_l)=W'(u_r)-W'(u_l),
        \end{equation} 
        which is the Rankine-Hugoniot equation \eqref{RH}. It cannot hold if $c=0$.
      \end{proof}

      We can now prove Proposition \ref{DiscreteSolitonsLin}: 
      \begin{proof}
	If $c=0$, Lemma \ref{StatWave} gives the desired result. Suppose $c\neq 0$.
        Using Fourier transform on \eqref{Soliton} implies:
        \begin{align*}
          c^2 \ksi^2 \mathcal{F}(\phi)(\ksi)=& \lt(\exp(i\ksi)-2+\exp(-i\ksi) \rt)\mathcal{F}(\phi)(\ksi).
        \end{align*}
        The equation:
        \begin{equation}
	  \label{MonSupport}
          c^2 \ksi^2 = 2\lt(1-\cos(\ksi)\rt),
        \end{equation}
        has a finite number of solutions $\ksi_j$, $j\in J$. Therefore, there exist $K_j \in \mathbb{N}$, $a_{jk} \in \R$ such that:
        \begin{equation*}
          \mathcal{F}(\phi)= \sum_{j=1}^J \sum_{k=0}^{K_j} a_{jk}\delta_{\ksi_j}^{(k)}.
        \end{equation*}
        Thus:
        \begin{equation*}
          \phi(x)=\sum_{j=1}^J \sum_{k=0}^{K_j} a_{jk}(ix)^{k} \exp(ix\ksi_j).
        \end{equation*}
        Since $\phi'$ has a limit for $x \rightarrow +\infty$ then $a_{jk}=0$ if $j \neq 0$ or $k>1$. Thus, there exists no discrete shock wave with $c \neq 0$.
      \end{proof}

   \subsection{Convex non-linear potential}
    
      We now prove Proposition \ref{DiscreteSolitonsNLin}.       
      
      \begin{proof}
        Suppose $u_l \neq u_r$. Test now \eqref{Soliton} with $\phi'$:
        \begin{align*}
          \frac{c^2}{2} \lt( u_r^2 -u_l^2 \rt)
          =& \lim_{R \rightarrow + \infty} \int_{-R}^R \lt\{W'(\phi(x+1)-\phi(x))-W'(\phi(x)-\phi(x-1)) \rt\} \phi'(x) dx\\
          =&\lim_{R\rightarrow + \infty} \Bigg\{ \int_{-R}^R W'(\phi(x+1)-\phi(x)) \lt(\phi'(x+1)-\phi'(x) \rt) dx\\
          &+\int_{R-1}^R \phi'(x+1) W'(\phi(x+1)-\phi(x)) dx - \int_{-R-1}^{-R} \phi'(x+1)W'(\phi(x+1)\\
          &-\phi(x)) dx\Bigg\}\\
          =& W(u_r)-W(u_l)+W'(u_r)u_r-W'(u_l)u_l.
        \end{align*}
        Using \eqref{RH2}, we get:
        \begin{equation}
          \label{Soliton2}
          \frac{1}{2} \lt(W'(u_r)-W'(u_l) \rt)= \frac{W(u_r)-W(u_l)}{u_r-u_l}.
        \end{equation} 
        Yet:
        \begin{align*}
          &\frac{1}{2} \lt(W'(u_r)-W'(u_l) \rt)- \frac{W(u_r)-W(u_l)}{u_r-u_l}
          \\&=
          \frac{1}{u_r-u_l} \int_{u_l}^{u_r} \lt\{ \frac{u_r-s}{u_r-u_l}W'(u_l)+\frac{s-u_l}{u_r-u_l} -W'(s) \rt\} ds,
        \end{align*}
        and as $W'$ is strictly convex:
        \begin{equation*}
          \frac{u_r-s}{u_r-u_l}W'(u_l)+\frac{s-u_l}{u_r-u_l} -W'(s)>0, \forall s \in ]u_l,u_r[.
        \end{equation*}
        This is contradictory. Therefore, as $u_l\neq u_r$, there does not exist any discrete shock wave of \eqref{Newton}.
      \end{proof}

  \section*{Acknowledgement}
  
    We wish to thank Claude Le Bris and Frédéric Legoll for their help and Gabriel Stoltz and Frédéric Lagoutière, for fruitful discussions.

\end{document}